\documentclass[3p, nopreprintline]{elsarticle}
\usepackage[T1]{fontenc}
\usepackage{amsfonts,amsmath,amsthm,amscd,amssymb,latexsym,amsbsy,caption}
\usepackage{tikz}
\usetikzlibrary{positioning, calc, trees, decorations.markings, patterns, arrows, arrows.meta}
\usepackage{diagmac2}

\biboptions{sort&compress}

\newcommand{\BB}{\mathbf{B}}
\newcommand{\RR}{\mathbb{R}}
\newcommand{\NN}{\mathbb{N}}

\newcommand{\levdist}{1.2cm}

\makeatletter

\renewcommand{\p@enumii}{}
\makeatother

\DeclareMathOperator{\card}{card}

\newtheorem{theorem}{Theorem}[section]
\newtheorem{lemma}[theorem]{Lemma}
\newtheorem{corollary}[theorem]{Corollary}
\newtheorem{proposition}[theorem]{Proposition}

\theoremstyle{definition}
\newtheorem{definition}[theorem]{Definition}
\newtheorem{problem}[theorem]{Problem}

\newtheorem{conjecture}{Conjecture}[section]

\theoremstyle{remark}
\newtheorem{example}[theorem]{Example}
\newtheorem{remark}[theorem]{Remark}

\numberwithin{equation}{section}
\allowdisplaybreaks[4]

\date{}
\journal{}

\begin{document}

\begin{frontmatter}

\title{Labeled trees generating complete, compact, and discrete ultrametric spaces}

\author[1]{Oleksiy Dovgoshey\corref{cor1}}
\ead{oleksiy.dovgoshey@gmail.com}

\author[2]{Mehmet K\"{u}\c{c}\"{u}kaslan}
\ead{mkucukaslan@mersin.edu.tr}

\address[1]{Department of Theory of Functions, Institute of Applied Mathematics and Mechanics of NASU, Dobrovolskogo str. 1, Slovyansk 84100, Ukraine}

\address[2]{Department of Mathematics, Faculty of Science and Arts, Mersin University, 33343 Mersin, Turkey}

\cortext[cor1]{Corresponding author}

\begin{abstract}
We investigate the interrelations between labeled trees and ultrametric spaces generated by these trees. The labeled trees, which generate complete ultrametrics, totally bounded ultrametrics, and discrete ones, are characterized up to isomorphism. As corollary, we obtain a characterization of labeled trees generating compact ultrametrics, and discrete totally bounded ultrametrics. It is also shown that every ultrametric space generated by labeled tree contains a dense discrete subspace.
\end{abstract}


\end{frontmatter}

\section{Introduction}

The following problem was raised in 2001 by I.~M.~Gelfand: Using graph theory describe up to isometry all finite ultrametric spaces~\cite{Lem2001}. An appropriate representation of finite, ultrametric spaces by monotone trees was proposed by V.~Gurvich and M.~Vyalyi in~\cite{GV2012DAM}. A simple geometric description of Gurvich---Vyalyi representing trees was found in \cite{PD2014JMS}. This description allows us effectively use the Gurvich---Vyalyi representation in various problems associated with finite ultrametric spaces. In particular, it leads to a graph-theoretic interpretation of the Gomory---Hu inequality \cite{DPT2015}. A characterization of finite ultrametric spaces which are as rigid as possible also was obtained \cite{DPT2017FPTA} on the basis of the Gurvich---Vyalyi representation. Some other extremal properties of finite ultrametric spaces and related them properties of monotone rooted trees have been found in~\cite{DP2020pNUAA}. The interconnections between the Gurvich---Vyalyi representation and the space of balls endowed with the Hausdorff metric are discussed in~\cite{Dov2019pNUAA} (see also \cite{Qiu2009pNUAA, Qiu2014pNUAA, DP2018pNUAA, Pet2018pNUAA, Pet2013PoI}).

The Gurvich---Vyalyi representing trees can be considered as a subclass of finite trees endowed with some special labeling on vertex set. The trees with labeled vertices are studied by many mathematicians and there are a number of interesting results in this directions. In survey~\cite{GalTEJoC2019}, J.~Gallian writes that over 200 graph labelings techniques have been studied in over 2800 paper during the past 50 years. In this regards, we only note that, in almost all studies of trees with labeled vertices, it is assumed that the trees are finite. The infinite trees endowed with positive real labelings on the set of edges are known as the so-called \(R\)-trees (see~\cite{Ber2019SMJ} for some interesting results related to \(R\)-trees and ultrametrics). A description of interrelations between finite subtrees of \(R\)-trees and finite, monotone rooted trees can be found in~\cite{Dov2020TaAoG}. The categorical equivalence of trees and ultrametric spaces was investigated in \cite{H04} and \cite{Lem2003AU}.

Motivated by Gurvich---Vyalyi representation of finite ultrametric spaces and some results of Bruhn, Diestel, Halin, K\"{u}hn, Pott, Spr\"{u}ssel, and Stein \cite{BDSJGT2005, BDCPC2006, BSCPC2010, DKEJC2004, DieJCTSB2006, DieDM2011, DSAM2011, DSTA2011, DSDM2012, DieAMSUH2017, DPJCTSB2017} on topological aspects of infinite graphs we consider infinite trees whose vertices are labeled by nonnegative real numbers and ultrametric spaces generated by such trees.

The paper is organized as follows. Section~\ref{sec2} and Section~\ref{sec3} contain some necessary concepts and facts from the theory of metric spaces and graph theory, respectively. In Section~\ref{sec4} we introduce into consideration the ultrametric spaces \((V(T), d_l)\) generated by non-degenerate vertex labelings \(l\) of arbitrary trees \(T\). The first main result of the paper is Theorem~\ref{t4.7} characterizing, up to isomorphism, the labeled trees \(T(l)\) for which the corresponding ultrametric spaces \((V(T), d_l)\) are complete. The characterizations of labeled trees generating discrete ultrametrics and totally bounded ones are found in Theorem~\ref{t4.11} and Theorem~\ref{t11.8}, respectively. Using these results we describe, up to isomorphism, the labeled trees generating discrete totally bounded ultrametrics in Theorem~\ref{t11.16}. The final result of Section~\ref{sec4} is Theorem~\ref{t4.18} characterizing the labeled trees \(T(l)\) for which the ultrametric spaces \((V(T), d_l)\) are compact. The last fifth section contains some conjectures and examples related to subject of the paper.

\emph{Concluding remarks}. The results obtained in the paper indicate a close connection between the combinatorial properties of an infinite tree and the properties of ultrametric spaces generated by labelings on its vertex set.
\begin{itemize}
\item A tree \(T\) is rayless if and only if every ultrametric generated by vertex labeling is complete (Corollary~\ref{c4.7}).
\item \(T\) is locally finite if and only if every ultrametric generated by vertex labeling is discrete (Corollary~\ref{c4.12}).
\item \(T\) is rayless, at most countable, and has no adjacent vertices of infinite degree if and only if there is vertex labeling generating a compact ultrametric (Theorem~\ref{t7}).
\item \(T\) is locally finite if and only if there is vertex labeling generating a discrete totally bounded ultrametric (Corollary~\ref{c4.18}).
\end{itemize}
It seems interesting to study similar problems for general infinite connected graphs using the spanning trees technique. Another promising direction of research is the study of ultrametric spaces generated by some special labelings. For example, we can consider the case when the label of a vertex depends on the degree of this vertex.

\section{Definitions and facts from theory of metric spaces}\label{sec2}

Let us start from basic concepts. In what follows, we will denote by \(\RR^{+}\) the half-open interval \([0, \infty)\) and write \(\NN\) for the set of all positive integers, \(\{1, 2, \ldots\}\).

A \textit{metric} on a set $X$ is a function $d\colon X\times X\rightarrow \RR^+$ such that for all \(x\), \(y\), \(z \in X\)
\begin{enumerate}
\item $d(x,y)=d(y,x)$,
\item $(d(x,y)=0)\Leftrightarrow (x=y)$,
\item \(d(x,y)\leq d(x,z) + d(z,y)\).
\end{enumerate}

A metric space \((X, d)\) is \emph{ultrametric} if the \emph{strong triangle inequality}
\[
d(x,y)\leq \max \{d(x,z),d(z,y)\}.
\]
holds for all \(x\), \(y\), \(z \in X\). In this case the function \(d\) is called \emph{an ultrametric} on \(X\).

\begin{definition}\label{d2.2}
Let \((X, d)\) and \((Y, \rho)\) be metric spaces. A bijective mapping \(\Phi \colon X \to Y\) is said to be an \emph{isometry} if
\[
d(x,y) = \rho(\Phi(x), \Phi(y))
\]
holds for all \(x\), \(y \in X\). The metric spaces are \emph{isometric} if there is an isometry of these spaces.
\end{definition}

Let \((X, d)\) be a metric space. An \emph{open ball} with a \emph{radius} \(r > 0\) and a \emph{center} \(c \in X\) is the set 
\[
B_r(c) = \{x \in X \colon d(c, x) < r\}.
\]

We denote by \(\mathbf{B}_X\) the set of all open balls in \((X, d)\).

A sequence \((x_n)_{n \in \mathbb{N}} \subseteq X\) is a \emph{Cauchy sequence} in \((X, d)\) if, for every \(r > 0\), there is an integer \(n_0 \in \NN\) such that \(x_n \in B_r(x_{n_0})\) for every \(n \geqslant n_0\). It is easy to see that \((x_n)_{n \in \mathbb{N}}\) is a Cauchy sequence if and only if 
\begin{equation*}
\lim_{n \to \infty} \sup\{d(x_n, x_{n+k}) \colon k \in \mathbb{N}\} = 0.
\end{equation*}

\begin{remark}\label{r2.17}
Here and later the symbol \((x_n)_{n\in \mathbb{N}} \subseteq X\) means that \(x_n \in X\) holds for every \(n \in \mathbb{N}\).
\end{remark}

There exists a comfortable ``ultrametric modification'' of the notion of Cauchy sequence (see, for example, \cite[p.~4]{PerezGarcia2010} or \cite[Theorem~1.6, Statement~(13)]{Comicheo2018}).

\begin{proposition}\label{p4.5}
Let \((X, d)\) be an ultrametric space. A sequence \((x_n)_{n \in \mathbb{N}} \subseteq X\) is a Cauchy sequence if and only if the limit relation
\begin{equation*}
\lim_{n \to \infty} d(x_n, x_{n+1}) = 0
\end{equation*}
holds.
\end{proposition}

A sequence \((x_n)_{n \in \mathbb{N}}\) of points in a metric space \((X, d)\) is said to be convergent to a point \(a \in X\),
\begin{equation*}
\lim_{n \to \infty} x_n = a,
\end{equation*}
if, for every open ball \(B\) containing \(a\), it is possible to find an integer \(n_0 \in \NN\) such that \(x_n \in B\) for every \(n \geqslant n_0\). Thus, \((x_n)_{n \in \mathbb{N}}\) is \emph{convergent} to \(a\) if and only if 
\begin{equation*}
\lim_{n \to \infty} d(x_n, a) = 0.
\end{equation*}
A sequence is convergent if it is convergent to some point. It is clear that every convergent sequence is a Cauchy sequence.

The next proposition follows, for example, from Theorem~6.8.3 in \cite{Sea2007}.

\begin{proposition}\label{p2.7}
Let \((X, d)\) be a metric space and let \((x_n)_{n \in \NN} \subseteq X\) be a Cauchy sequence in \((X, d)\). Then \((x_n)_{n \in \NN}\) is convergent if and only if it has a convergent subsequence.
\end{proposition}

Now we present a definition of \emph{total boundedness}.

\begin{definition}\label{d2.10}
A subset \(A\) of a metric space \((X, d)\) is totally bounded if for every \(r > 0\) there is a finite set \(\{B_r(x_1), \ldots, B_r(x_n)\} \subseteq \mathbf{B}_X\) such that
\[
A \subseteq \bigcup_{i = 1}^{n} B_r(x_i).
\] 
\end{definition}

There exists a simple interdependence between the total boundedness of a set \(A \subseteq X\) and Cauchy sequences in \(A\).

\begin{proposition}\label{p2.11}
A subset \(A\) of a metric space \((X, d)\) is totally bounded if and only if every sequence of points of \(A\) contains a Cauchy subsequence. 
\end{proposition}

See, for example, Theorem~7.8.2 \cite{Sea2007}.

\begin{corollary}\label{c2.17}
Let \((X, d)\) be a metric space. If \(A \subseteq X\) is totally bounded in \((X, d)\) and \(C\) is a subset of \(A\), then \(C\) is totally bounded in \((A, d|_{A \times A})\).
\end{corollary}

The next basic for us concept is the concept of \emph{completeness}.

\begin{definition}\label{d2.6}
A metric space \((X, d)\) is complete if for every Cauchy sequence \((x_n)_{n \in \mathbb{N}} \subseteq X\) there is a point \(a \in X\) such that 
\[
\lim_{n \to \infty} x_n = a.
\]
\end{definition}

Thus, a metric space is complete if and only if the set of Cauchy sequences coincides with the set of convergent sequences in this space.

An important subclass of complete metric spaces is the class of \emph{compact} metric spaces.

\begin{definition}[Borel---Lebesgue property]\label{d2.3}
Let \((X, d)\) be a metric space. A subset \(A\) of \(X\) is compact if every family \(\mathcal{F} \subseteq \mathbf{B}_X\) satisfying the inclusion
\[
A \subseteq \bigcup_{B \in \mathcal{F}} B
\]
contains a finite subfamily \(\mathcal{F}_0 \subseteq \mathcal{F}\) such that 
\[
A \subseteq \bigcup_{B \in \mathcal{F}_0} B.
\]
\end{definition}

A standard definition of compactness usually formulated as: Every open cover of \(A\) in \(X\) has a finite subcover.

The following classical theorem was proved by Frechet and it is a ``compact'' analog of Proposition~\ref{p2.11}. 

\begin{proposition}[Bolzano---Weierstrass property]\label{p2.9}
A subset \(A\) of a metric space is compact if and only if every sequence of points of \(A\) contains a subsequence which converges to a point of \(A\).
\end{proposition}

The next corollary shows that the class of compact metric spaces is the intersection of the class of complete metric spaces with the class of totally bounded ones.

\begin{corollary}[Spatial Criterion]\label{c2.9}
A metric space is compact if and only if this space is complete and totally bounded.
\end{corollary}

This and other criteria of compactness can be found, for example, in \cite[p.~206]{Sea2007}.

Let \((X, d)\) be a metric space and let \(S \subseteq X\). The set \(S\) is said to be \emph{dense} in \((X, d)\) if for every \(a \in X\) there is a sequence \((s_n)_{n\in \mathbb{N}} \subseteq S\) such that
\[
a = \lim_{n \to \infty} s_n.
\]

Recall that a point \(p\) of a metric space \((X, d)\) is \emph{isolated} if there is \(\varepsilon > 0\) such that \(d(p, x) > \varepsilon\) for every \(x \in X \setminus \{p\}\). If \(p\) is not an isolated point of \(X\), then \(p\) is called an \emph{accumulation point} of \(X\). We say that a set \(A \subseteq X\) is \emph{discrete} if all points of \(A\) are isolated.

It will be shown in Proposition~\ref{p4.13} of Section~\ref{sec4} that every ultrametric space, generated by labeled graph, contains a dense discrete subset.

\section{Definitions and facts from graph theory}\label{sec3}

A \textit{graph} is a pair $(V,E)$ consisting of a set $V$ and a set $E$ whose elements are unordered pairs \(\{u, v\}\) of different points \(u\), \(v \in V\). For a graph $G=(V,E)$, the sets \(V=V(G)\) and \(E=E(G)\) are called \textit{the set of vertices} and \textit{the set of edges}, respectively. A graph \(G\) is \emph{finite} if \(V(G)\) is a finite set. If \(\{x,y\} \in E(G)\), then the vertices \(x\) and \(y\) are called \emph{adjacent}. In what follows we will always assume that \(E(G) \cap V(G) = \varnothing\).

Recall that a graph \(G_1\) is a \emph{subgraph} of a graph \(G\) if
\[
V(G_1) \subseteq V(G) \quad \text{and} \quad E(G_1) \subseteq E(G).
\]
In this case we will write \(G_1 \subseteq G\). If \(\{G_i \colon i \in I\}\) is a family of subgraphs of a graph \(G\), then, by definition, the union \(\bigcup_{i \in I} G_i\) is a subgraph \(G^{*}\) of \(G\) such that 
\[
V(G^{*}) = \bigcup_{i \in I} V(G_i) \quad \text{and} \quad E(G^{*}) = \bigcup_{i \in I} E(G_i).
\]
Similarly, the intersection \(\bigcap_{i \in I} G_i\) is a subgraph \(G_{*}\) of \(G\) with
\begin{equation}\label{e3.4}
V(G_{*}) = \bigcap_{i \in I} V(G_i) \quad \text{and} \quad E(G_{*}) = \bigcap_{i \in I} E(G_i).
\end{equation}

Let \(v^*\) be a vertex of a graph \(G\). The \emph{neighborhood} \(N(v^*) = N_G(v^{*})\) is a subgraph of \(G\) induced by all vertices adjacent to \(v^{*}\). Thus, we have
\begin{align*}
V(N(v^{*})) &= \{u \in V(G) \colon \{u, v^{*}\} \in E(G)\},\\
E(N(v^{*})) &= \bigl\{\{u, v\} \in E(G) \colon u, v \in V(N(v^{*}))\bigr\}
\end{align*}
for every graph \(G\) and each \(v^{*} \in V(G)\). Let \(k\) be a cardinal number. The vertex \(v\) of a graph \(G\) has \emph{degree} \(k\) if 
\[
k = \card V(N(v)).
\]
The degree of \(v\) will be denoted as \(\delta_{G}(v)\) or simply as \(\delta(v)\).

A \emph{path} is a finite graph \(P\) whose vertices can be numbered without repetitions so that
\begin{equation}\label{e3.3}
V(P) = \{x_1, \ldots, x_k\} \quad \text{and} \quad E(P) = \{\{x_1, x_2\}, \ldots, \{x_{k-1}, x_k\}\}
\end{equation}
with \(k \geqslant 2\). We will write \(P = (x_1, \ldots, x_k)\) or \(P = P_{x_1, x_k}\) if \(P\) is a path satisfying \eqref{e3.3} and said that \(P\) is a \emph{path joining \(x_1\) and \(x_k\)}. A graph \(G\) is \emph{connected} if for every two distinct vertices of \(G\) there is a path \(P \subseteq G\) joining these vertices.

A finite graph \(C\) is a \textit{cycle} if there exists an enumeration of its vertices without repetition such that \(V(C) = \{x_1, \ldots, x_n\}\)  and 
\[
E(C) = \{\{x_1, x_2\}, \ldots, \{x_{n-1}, x_n\}, \{x_n, x_1\}\} \quad \text{with } n \geqslant 3.
\]

\begin{definition}\label{d9.8}
A connected graph \(T\) with \(V(T) \neq \varnothing\) and without cycles is called a \emph{tree}.
\end{definition}

A tree \(T\) is \emph{locally finite} if the inequality \(\delta(v) < \infty\) holds for every \(v \in V(T)\).

We shall say that a tree \(T\) is a \emph{star} if there is a vertex \(c \in V(T)\), the \emph{center} of \(T\), such that \(c\) and \(v\) are adjacent for every \(v \in V(T) \setminus \{c\}\).

An infinite graph \(G\) of the form 
\[
V(G) = \{x_1, x_2 \ldots, x_n, x_{n+1}, \ldots\}, \quad E(G) = \{\{x_1, x_2\}, \ldots \{x_n, x_{n+1}\}, \ldots\},
\]
where all \(x_n\) are assumed to be distinct, is called a \emph{ray} \cite{Diestel2017}. It is clear that every ray is a tree. A graph is \emph{rayless} if it contains no rays.

\begin{proposition}\label{p3.2}
Every infinite connected graph has a vertex of infinite degree or contains a ray.
\end{proposition}

For the proof see Proposition~8.2.1 in \cite{Diestel2017}.

The following statement is well known for finite trees (see, for example, Proposition~4.1 \cite{BM2008}) and is usually considered self-evident for infinite trees. 

\begin{lemma}\label{l9.14}
In each tree, every two different vertices are connected by exactly one path.
\end{lemma}

\begin{proof}
If \(T\) is an infinite tree and \(v_1\), \(v_2\) are two different vertices of \(T\) connected by some paths \(P_1 \subseteq T\) and \(P_2 \subseteq T\), then the graph \(P_1 \cup P_2\) is a finite connected subgraph of \(T\). Since \(T\) does not have cycles, \(P_1 \cup P_2\) is also acyclic. Hence, \(P_1 \cup P_2\) is a finite tree and \(P_1\), \(P_2\) are paths connected \(v_1\) and \(v_2\) in \(P_1 \cup P_2\). Thus, \(P_1 = P_2\) holds.
\end{proof}

In the next definition we introduce an analogue of convex hull for arbitrary trees.

\begin{definition}\label{d3.6}
Let \(T\) be a tree and let \(A\) be a nonempty subset of \(V(T)\). A subtree \(H_A\) of the tree \(T\) is the \emph{hull} of \(A\) if \(A \subseteq V(H_A)\) and, for every subtree \(T^{*}\) of \(T\), the tree \(H_A\) is a subtree of \(T^{*}\) whenever \(A \subseteq V(T^*)\).
\end{definition}

Thus, \(H_A\) is the smallest subtree of \(T\) which contains \(A\). We want to make sure that for every tree \(T\) and each nonempty \(A \subseteq V(T)\) the hull \(H_A\) is well defined and unique.

\begin{proposition}\label{p3.7}
Let \(T\) be a tree, \(A\) be a nonempty subset of \(V(T)\) and let \(\mathcal{F}_A\) be the set of all subtrees \(T^{*}\) of \(T\) for which the inclusion \(A \subseteq V(T^{*})\) holds. Then the graph \(\bigcap_{T^{*} \in \mathcal{F}_A} T^{*}\) is the hull of \(A\),
\begin{equation}\label{p3.7:e1}
H_A = \bigcap_{T^{*} \in \mathcal{F}_A} T^{*}.
\end{equation}
\end{proposition}

\begin{proof}
It is clear that \(\bigcap_{T^{*} \in \mathcal{F}_A} T^{*}\) is a subgraph of \(T^{*}\) for every \(T^{*} \in \mathcal{F}_A\). In particular, \(\bigcap_{T^{*} \in \mathcal{F}_A} T^{*}\) is a subgraph of \(T\) because \(T \in \mathcal{F}_A\). Since \(T\) is a tree, the subgraph \(\bigcap_{T^{*} \in \mathcal{F}_A} T^{*}\) contains no cycles. Hence, to prove \eqref{p3.7:e1} it suffices to show that \(\bigcap_{T^{*} \in \mathcal{F}_A} T^{*}\) is connected.

Let \(u\) and \(v\) be distinct vertices of \(\bigcap_{T^{*} \in \mathcal{F}_A} T^{*}\) and let \(P_{u, v}\) be the path joining \(u\) and \(v\) in \(T\). Then \(u\) and \(v\) belong to \(V(T^{*})\) for every \(T^{*} \in \mathcal{F}_A\). Using Lemma~\ref{l9.14} we obtain \(P_{u, v} \subseteq T^{*}\) for every \(T^{*} \in \mathcal{F}_A\). From~\eqref{e3.4} with \(\mathcal{F} = \mathcal{F}_A\) it follows that the path \(P_{u, v}\) is also a subgraph of \(\bigcap_{T^{*} \in \mathcal{F}_A} T^{*}\). Thus, \(\bigcap_{T^{*} \in \mathcal{F}_A} T^{*}\) is connected as required.
\end{proof}

\begin{example}\label{ex3.7}
Let \(T\) be an infinite tree, let a ray \(R\),
\[
V(R) = \{r_1, r_2, \ldots, r_n, r_{n+1}, \ldots\}, \quad E(R) = \{\{r_1, r_2\}, \ldots, \{r_n, r_{n+1}\}, \ldots\},
\]
be a subgraph of \(T\), and \(v\) be a vertex of \(T\) such that \(v \notin V(R)\). We claim that there is a unique \(n(v) \in \NN\) such that \(r_{n(v)}\) is the only common vertex of \(R\) and of the path \(P_{r_{n(v)}, v}\) joining \(r_{n(v)}\) and \(v\) in \(T\),
\begin{equation}\label{ex3.7:e1}
\{r_{n(v)}\} = V(R) \cap V(P_{r_{n(v)}, v}).
\end{equation}
We first prove the existence of some \(n(v) \in \NN\) satisfying \eqref{ex3.7:e1} and show that the graph \(R \cup P_{r_{n(v)}, v}\) is the hull in \(T\) of the set \(A \stackrel{\text{def}}{=} V(R) \cup \{v\}\),
\begin{equation*}
H_A = R \cup P_{r_{n(v)}, v}.
\end{equation*}
Let \(v \in V(T) \setminus V(R)\) be fixed. To find \(n(v) \in \NN\) satisfying \eqref{ex3.7:e1} it suffices to consider an arbitrary \(u \in V(R)\) and the path \((u^1, \ldots, u^m) \subseteq T\) with \(u^1 = u\) and \(u^m = v\). Since \(u \in V(R)\) and \(v \notin V(R)\) hold, there is \(m_1 \in \{1, \ldots, m-1\}\) such that \(u^{m_1} \in V(R)\) and \(u^{j} \notin V(R)\) whenever \(j \in \{m_1+1, \ldots, m-1\}\). Consequently, there is \(n_1 \in \NN\) such that \(r_{n_1} = u^{m_1}\). If we set \(n(v) \stackrel{\text{def}}{=} n_1\), then \eqref{ex3.7:e1} holds with \(P_{r_{n(v)}, v} \stackrel{\text{def}}{=} (u^{m_1}, \ldots, u^{m})\). Since \(P_{r_{n(v)}, v}\) and \(R\) are connected and have the common vertex \(r_{n(v)}\), the union \(R \cup P_{r_{n(v)}, v}\) is a subtree of \(T\). It is also clear that
\[
A \subseteq V(R \cup P_{r_{n(v)}, v})
\]
holds. Now Definition~\ref{d3.6} implies that \(H_A\) is a subtree of \(R \cup P_{r_{n(v)}, v}\). If we have 
\[
H_A \neq R \cup P_{r_{n(v)}, v},
\]
then there is \(j \in \{m_1+1, \ldots, m-1\}\) such that \(u^{j} \notin V(H_A)\). Lemma~\ref{l9.14} and \(u^{j} \notin V(H_A)\) imply that \(H_A\) is disconnected, contrary to Definition~\ref{d3.6}. From~Proposition~\ref{p3.7} it follows that the hull \(H_A\) is unique. Consequently, the number \(n(v) \in \NN\) satisfying \eqref{ex3.7:e1} is also unique. 

In what follows we will say that \(R \cup P_{r_{n(v)}, v}\) is a \emph{comb} in \(T\), \(v\) is a \emph{tooth} of this comb, and \(r_{n(v)}\) is the \emph{root} of the tooth \(v\) (see Figure~\ref{fig2}).
\end{example}

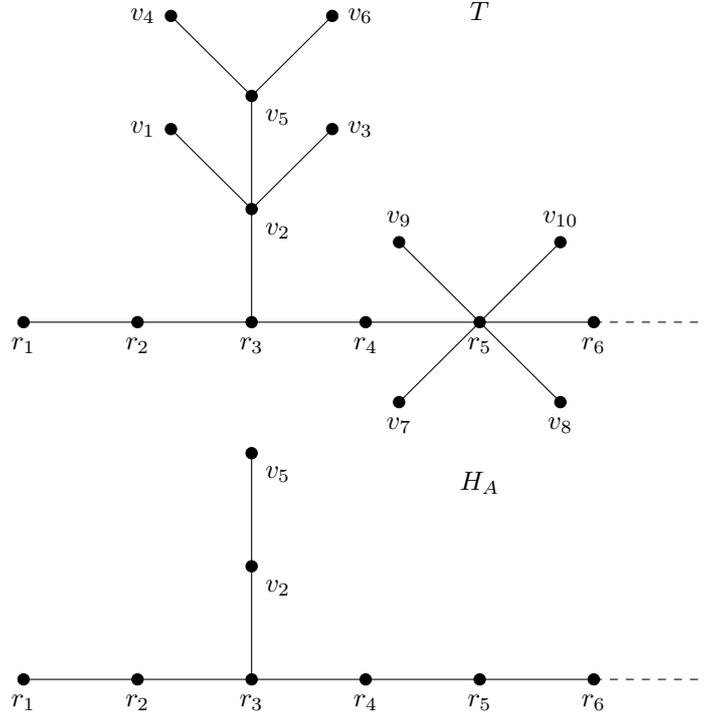
\begin{figure}[ht]
\begin{center}
\begin{tikzpicture}[
level 1/.style={level distance=\levdist,sibling distance=24mm},
level 2/.style={level distance=\levdist,sibling distance=12mm},
level 3/.style={level distance=\levdist,sibling distance=6mm},
solid node/.style={circle,draw,inner sep=1.5,fill=black},
micro node/.style={circle,draw,inner sep=0.5,fill=black}]

\def\dx{1.5cm}
\node at (5*\dx, 3*\dx) [label= below:\(T\)] {};
\draw (\dx, 0) -- (6*\dx, 0);
\draw [dashed] (6*\dx, 0)  --  (7*\dx, 0);
\node [solid node] at (\dx, 0) [label= below:\(r_1\)] {};
\node [solid node] at (2*\dx, 0) [label= below:\(r_2\)] {};
\node [solid node] at (3*\dx, 0) [label= below:\(r_3\)] {};
\node [solid node] at (4*\dx, 0) [label= below:\(r_4\)] {};
\node [solid node] at (5*\dx, 0) [label= below:\(r_5\)] {};
\node [solid node] at (6*\dx, 0) [label= below:\(r_6\)] {};

\draw (3*\dx, 0) -- (3*\dx, 2*\dx);
\node [solid node] at (3*\dx, \dx) [label= below right:\(v_2\)] {};
\node [solid node] at (3*\dx, 2*\dx) [label= below right:\(v_5\)] {};

\draw (3*\dx, \dx) -- ($(3*\dx, \dx) + (135:\dx)$) node [solid node, label= left:\(v_1\)] {};
\draw (3*\dx, \dx) -- ($(3*\dx, \dx) + (45:\dx)$) node [solid node, label= right:\(v_3\)] {};
\draw (3*\dx, 2*\dx) -- ($(3*\dx, 2*\dx) + (135:\dx)$) node [solid node, label= left:\(v_4\)] {};
\draw (3*\dx, 2*\dx) -- ($(3*\dx, 2*\dx) + (45:\dx)$) node [solid node, label= right:\(v_6\)] {};

\draw (5*\dx, 0) -- ($(5*\dx, 0) + (135:\dx)$) node [solid node, label= above:\(v_9\)] {};
\draw (5*\dx, 0) -- ($(5*\dx, 0) + (45:\dx)$) node [solid node, label= above:\(v_{10}\)] {};
\draw (5*\dx, 0) -- ($(5*\dx, 0) + (-135:\dx)$) node [solid node, label= below:\(v_7\)] {};
\draw (5*\dx, 0) -- ($(5*\dx, 0) + (-45:\dx)$) node [solid node, label= below:\(v_{8}\)] {};
\end{tikzpicture}

\begin{tikzpicture}[
level 1/.style={level distance=\levdist,sibling distance=24mm},
level 2/.style={level distance=\levdist,sibling distance=12mm},
level 3/.style={level distance=\levdist,sibling distance=6mm},
solid node/.style={circle,draw,inner sep=1.5,fill=black},
micro node/.style={circle,draw,inner sep=0.5,fill=black}]

\def\dx{1.5cm}
\node at (5*\dx, 2*\dx) [label= below:\(H_A\)] {};
\draw (\dx, 0) -- (6*\dx, 0);
\draw [dashed] (6*\dx, 0)  --  (7*\dx, 0);
\node [solid node] at (\dx, 0) [label= below:\(r_1\)] {};
\node [solid node] at (2*\dx, 0) [label= below:\(r_2\)] {};
\node [solid node] at (3*\dx, 0) [label= below:\(r_3\)] {};
\node [solid node] at (4*\dx, 0) [label= below:\(r_4\)] {};
\node [solid node] at (5*\dx, 0) [label= below:\(r_5\)] {};
\node [solid node] at (6*\dx, 0) [label= below:\(r_6\)] {};

\draw (3*\dx, 0) -- (3*\dx, 2*\dx);
\node [solid node] at (3*\dx, \dx) [label= below right:\(v_2\)] {};
\node [solid node] at (3*\dx, 2*\dx) [label= below right:\(v_5\)] {};
\end{tikzpicture}
\end{center}
\caption{The hull \(H_A\) of \(A = \{v_5\} \cup \{r_i \colon i \in \NN\}\) is comb in \(T\), the vertex \(r_3\) is the root of the tooth \(v_5\) in this comb.} 
\label{fig2}
\end{figure}

\begin{remark}\label{r3.8}
Thus, we always assume that each comb has exactly one tooth with exactly one root. The combs with a large number of teeth are more often used in theory of ultrametric spaces and graph theory (see, for example, the Comb representation of compact ultrametric spaces \cite{LamBr2017} or the Star-Comb Lemma~\cite[Lemma~8.2.2]{Diestel2017}).
\end{remark}

Let us recall the concept of labeled trees.

\begin{definition}\label{d2.4}
A labeled tree is a pair \((T, l)\), where \(T\) is a tree and \(l\) is a mapping defined on the set \(V(T)\).

We say that \(T\) is a \emph{free tree} corresponding to \((T, l)\) and write \(T = T(l)\) instead of \((T, l)\). In what follows, we will consider only the nonnegative real-valued labelings \(l\colon V(T)\to \RR^{+}\).
\end{definition}

Before introducing into consideration the concept of isomorphism of labeled trees, it is useful to remind the definition of isomorphism for free trees.

\begin{definition}\label{d8.7}
Let $T_1$ and $T_2$ be trees. A bijection $f\colon V(T_1)\to V(T_2)$ is an \emph{isomorphism} of $T_1$ and $T_2$ if
\begin{equation*}
(\{u,v\} \in E(T_1)) \Leftrightarrow (\{f(u),f(v)\} \in E(T_2))
\end{equation*}
is valid for all $u$, $v \in V(T_1)$. Two trees are \emph{isomorphic} if there exists an isomorphism of these trees.
\end{definition}

For the case of labeled trees Definition~\ref{d8.7} must be modified as follows.

\begin{definition}\label{d2.5}
Let \(T_i=T_i(l_i)\) be a labeled tree, \(i = 1\), \(2\). A mapping $f\colon V(T_1) \to V(T_2)$ is an isomorphism of \(T_1(l_1)\) and \(T_2(l_2)\) if it is an isomorphism of the free trees $T_1$ and $T_2$ and the equality
\begin{equation*}
l_2(f(v))=l_1(v)
\end{equation*}
holds for every $v \in V(T_1)$.
\end{definition}

\section{Ultrametrics generated by labeled trees}\label{sec4}

Let \(T = T(l)\) be a labeled tree. Following~\cite{Dov2020TaAoG}, we define a mapping \(d_l \colon V(T) \times V(T) \to \RR^{+}\) as
\begin{equation}\label{e11.3}
d_l(u, v) = \begin{cases}
0 & \text{if } u = v,\\
\max\limits_{v^{*} \in V(P)} l(v^{*}) & \text{if } u \neq v,
\end{cases}
\end{equation}
where \(P\) is the path joining \(u\) and \(v\) in \(T\).

\begin{remark}\label{r11.10}
The correctness of this definition follows from Lemma~\ref{l9.14}.
\end{remark}

To formulate the first theorem of this section we recall the concept of \emph{pseudoultrametric space}. Let \(X\) be a set. A symmetric mapping \(d \colon X \times X \to \RR^{+}\) is a \emph{pseudoultrametric} on \(X\) if
\[
d(x, x) = 0 \quad \text{and} \quad d(x, y) \leqslant \max\{d(x, z), d(z, y)\}
\]
hold for all \(x\), \(y\), \(z \in X\). Every ultrametric is a pseudoultrametric, but a pseudoultrametric \(d\) on a set \(X\) is an ultrametric if and only if \(d(x, y) > 0\) holds for all distinct \(x\), \(y \in X\).

The notion of isometries can be extended on pseudoultrametrics as follows. If \((X, d)\) and \((Y, \rho)\) are pseudoultrametric spaces, then a mapping \(\Phi \colon X \to Y\) is an isometry of \((X, d)\) and \((Y, \rho)\) if \(\Phi\) is bijective and
\[
d(x, y) = \rho(\Phi(x), \Phi(y))
\]
holds for all \(x\), \(y \in X\).

\begin{theorem}\label{t11.9}
Let \(T = T(l)\) be a labeled tree. Then \((V(T), d_l)\) is a pseudoultrametric space. The function \(d_l\) is an ultrametric on \(V(T)\) if and only if the inequality
\begin{equation*}
\max\{l(u_1), l(v_1)\} > 0
\end{equation*}
holds for every \(\{u_1, v_1\} \in E(T)\).
\end{theorem}

A proof of Theorem~\ref{t11.9} can be obtained by simple modification of the proof of Proposition~3.2 \cite{Dov2020TaAoG}.

\begin{proposition}\label{l11.12}
Let \(T_1 = T_1(l_1)\) and \(T_2 = T_2(l_2)\) be labeled trees and let \(f \colon V(T_1) \to V(T_2)\) be an isomorphism of these trees. Then the equality
\[
d_{l_1}(u, v) = d_{l_2}(f(u), f(v))
\]
holds for all \(u\), \(v \in V(T_1)\).
\end{proposition}

\begin{proof}
It directly follows from Definition~\ref{d2.5} and formula~\eqref{e11.3}, because if \((v_1, \ldots, v_n)\) is a path joining some distinct \(v = v_1\) and \(u = v_n\) in \(T_1\), then \(f(u) \neq f(v)\) holds and \((f(v_1), \ldots, f(v_n))\) is a path joining \(f(v)\) and \(f(u)\) in \(T_2\) and we have the equality 
\[
\{l_1(v_1), \ldots, l_1(v_n)\} = \{l_2(f(v_1)), \ldots, l_2(f(v_n))\}.
\]
\end{proof}

\begin{corollary}\label{c11.6}
Let \(T_1 = T_1(l_1)\) and \(T_2 = T_2(l_2)\) be isomorphic labeled trees. Then the pseudoultrametric spaces \((V(T_1), d_{l_1})\) and \((V(T_2), d_{l_2})\) are isometric.
\end{corollary}

The converse statement is not valid in general. The following example is a modification of Example~3.1 \cite{Dov2020TaAoG}.

\begin{example}\label{ex11.13}
Let \(V = \{v_0, v_1, v_2, v_3, v_4\}\) be a five-point set, and let \(S = S(l_S)\) and \(P = P(l_P)\) be a labeled star with the center \(v_0\) and, respectively, a labeled path such that \(V(S) = V(P) = V\), \(l_S(v_0) = l_P(v_0) = 1\) and
\[
l_S(v_i) + 1 = l_P(v_i) = 1
\]
for \(i = 1\), \(\ldots\), \(4\) (see Figure~\ref{fig11}). Then, for all distinct \(u\), \(v \in V\), we have
\[
d_{l_P}(u, v) = d_{l_S}(u, v) = 1.
\]
Thus, the ultrametric spaces \((V, d_{l_P})\) and \((V, d_{l_S})\) coincide, but \(P(l_P)\) and \(S(l_S)\) are not isomorphic as labeled trees or even as free trees.

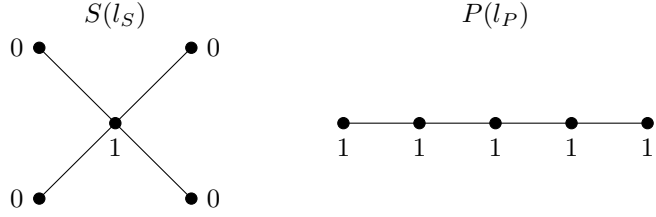
\begin{figure}[ht]
\begin{center}
\begin{tikzpicture}[
level 1/.style={level distance=\levdist,sibling distance=24mm},
level 2/.style={level distance=\levdist,sibling distance=12mm},
level 3/.style={level distance=\levdist,sibling distance=6mm},
solid node/.style={circle,draw,inner sep=1.5,fill=black},
micro node/.style={circle,draw,inner sep=0.5,fill=black}]

\def\dx{1cm}
\node at (0, \dx) [label= above:\(S(l_S)\)] {};
\node [solid node] at (0, 0) [label= below:\(1\)] {};
\node [solid node] at (\dx, \dx) [label= right:\(0\)] {};
\node [solid node] at (\dx, -\dx) [label= right:\(0\)] {};
\node [solid node] at (-\dx, \dx) [label= left:\(0\)] {};
\node [solid node] at (-\dx, -\dx) [label= left:\(0\)] {};
\draw (0, 0) -- (\dx, \dx);
\draw (0, 0) -- (\dx, -\dx);
\draw (0, 0) -- (-\dx, \dx);
\draw (0, 0) -- (-\dx, -\dx);

\node at (5*\dx, \dx) [label= above:\(P(l_P)\)] {};
\node [solid node] at (3*\dx, 0) [label= below:\(1\)] {};
\node [solid node] at (4*\dx, 0) [label= below:\(1\)] {};
\node [solid node] at (5*\dx, 0) [label= below:\(1\)] {};
\node [solid node] at (6*\dx, 0) [label= below:\(1\)] {};
\node [solid node] at (7*\dx, 0) [label= below:\(1\)] {};
\draw (3*\dx, 0) -- (7*\dx, 0);
\end{tikzpicture}
\end{center}
\caption{The star \(S\) and the path \(P\) are not isomorphic as trees, but the labelings \(l_S\) and \(l_P\) generate the same ultrametric on \(V\).}
\label{fig11}
\end{figure}
\end{example}

Example~\ref{ex11.13} shows that the properties of ultrametric spaces generated by different labeled trees can be the same. Thus, the following problem naturally arises.

\begin{problem}\label{pr4.6}
Let \(\mathcal{UP}\) be the class of ultrametric spaces with a given property \(\mathcal{P}\). What common properties do the labeled trees \(T = T(l)\) generating \((V(T), d_l) \in \mathcal{UP}\) have?
\end{problem}

Below we will consider this problem in the following cases:
\begin{itemize}
\item  \(\mathcal{P} =\) completeness,
\item \(\mathcal{P} =\) discreteness,
\item \(\mathcal{P} =\) total boundedness,
\item \(\mathcal{P} =\) discreteness + total boundedness,
\item \(\mathcal{P} =\) compactness,
\end{itemize}
and in each of these cases we find the corresponding characteristic properties of generating labeled trees.

Let us start from the completeness. 

In what follows we shall say that a labeling \(l \colon V(T) \to \mathbb{R}^{+}\) is \emph{non-degenerate} if the inequality
\[
\max\{l(u), l(v)\} > 0
\]
holds for every \(\{u, v\} \in E(T)\). 

\begin{lemma}\label{l4.7}
Let \(R\) be a ray, \(V(R) = \{v_1, v_2, \ldots, v_n, v_{n+1}, \ldots\}\),
\begin{equation}\label{l4.7:e1}
E(R) = \{\{v_1, v_2\}, \ldots, \{v_n, v_{n+1}\}, \ldots\},
\end{equation}
and let \(l \colon V(R) \to \RR^{+}\) be a non-degenerate labeling. The sequence \((v_n)_{n \in \NN}\) is a Cauchy sequence in the ultrametric space \((V(R), d_l)\) if and only if the limit relation
\begin{equation}\label{l4.7:e2}
\lim_{n \to \infty} l(v_n) = 0
\end{equation}
holds.
\end{lemma}

\begin{proof}
By Proposition~\ref{p4.5}, \((v_n)_{n \in \NN}\) is a Cauchy sequence in \((V(R), d_l)\) if and only if
\begin{equation}\label{l4.7:e3}
\lim_{n \to \infty} d_l(v_n, v_{n+1}) = 0.
\end{equation}
Using \eqref{e11.3} and \eqref{l4.7:e1} we can rewrite \eqref{l4.7:e3} as
\[
\lim_{n \to \infty} \left(\max\{l(v_n), l(v_{n+1})\}\right) = 0.
\]
Since all \(l(v_n)\) belong to \(\RR^{+}\), \eqref{l4.7:e2} holds if and only if 
\[
\limsup_{n \to \infty} l(v_n) = 0.
\]
Similarly, \eqref{l4.7:e3} is equivalent to
\[
\limsup_{n \to \infty} (\max\{l(v_n), l(v_{n+1})\}) = 0.
\]
Now using the equality
\[
\limsup_{n \to \infty} l(v_n) = \limsup_{n \to \infty} (\max\{l(v_n), l(v_{n+1})\})
\]
we see that \eqref{l4.7:e2} and \eqref{l4.7:e3} are equivalent.
\end{proof}

The next lemma will be useful to prove Theorem~\ref{t4.7}.

\begin{lemma}\label{l4.8}
Let \(R = R(l)\) be a labeled ray,
\[
V(R) = \{v_1, v_2, \ldots, v_n, v_{n+1}, \ldots\}, \quad E(R) = \{\{v_1, v_2\}, \ldots, \{v_n, v_{n+1}\}, \ldots\},
\]
with non-degenerate labeling. Then the sequence \((v_n)_{n \in \NN}\) is a Cauchy sequence in \((V(R), d_l)\) if and only if \((v_n)_{n \in \NN}\) contains a subsequence which is Cauchy in \((V(R), d_l)\).
\end{lemma}

\begin{proof}
If \((v_n)_{n \in \NN}\) is a Cauchy sequence, then \((v_n)_{n \in \NN}\) is a Cauchy subsequence of itself. 

Conversely, let \((v_{n_k})_{k \in \NN}\),
\begin{equation*}
1 \leqslant n_1 < n_2 < \ldots < n_k < n_{k+1} < \ldots,
\end{equation*}
be a Cauchy subsequence of \((v_n)_{n \in \NN}\). It is easy to see that, for every \(m \geqslant n_1\), there is the unique \(k(m) \in \NN\) such that 
\begin{equation}\label{l4.8:e2}
n_{k(m)} \leqslant m < n_{k(m)+1}.
\end{equation}

Let us denote by \(P_{v_{n_{k(m)}}, v_{n_{k(m)+1}}}\) the path joining \(v_{n_{k(m)}}\) and \(v_{n_{k(m)+1}}\) in \(R\). From \eqref{l4.8:e2} it follows that 
\begin{equation}\label{l4.8:e3}
v_m \in V(P_{v_{n_{k(m)}}, v_{n_{k(m)+1}}}).
\end{equation}
Now using \eqref{e11.3} and \eqref{l4.8:e3} we obtain
\begin{equation}\label{l4.8:e4}
l(v_m) \leqslant \max\{l(v) \colon v \in V(P_{v_{n_{k(m)}}, v_{n_{k(m)+1}}})\} = d_l(v_{n_{k(m)}}, v_{n_{k(m)+1}}).
\end{equation}
It is clear that the mapping 
\[
\{n_1, n_1+1, \ldots\} \ni m \mapsto n_{k(m)} \in \NN
\]
is increasing and satisfies the limit relation
\begin{equation}\label{l4.8:e6}
\lim_{m \to \infty} n_{k(m)} = +\infty.
\end{equation}
Proposition~\ref{p4.5}, \eqref{l4.8:e4} and \eqref{l4.8:e6} imply
\[
\limsup_{m \to \infty} l(v_m) \leqslant \limsup_{m \to \infty} d_l(v_{n_{k(m)}}, v_{n_{k(m)+1}}) = 0.
\]
Thus, we have
\begin{equation*}
\lim_{m \to \infty} l(v_m) = 0,
\end{equation*}
because \(l(v_m) \in \RR^{+}\) for every \(n \in \NN\). Using the last limit relation we obtain that \((v_n)_{n \in \NN}\) is a Cauchy sequence by Lemma~\ref{l4.7}.
\end{proof}

\begin{lemma}\label{l4.6}
Let \(T = T(l)\) be a labeled tree with non-degenerate labeling \(l \colon V(T) \to \RR^{+}\) and let \(T_1\) be a subtree of \(T\) having the labeling \(l_1 \colon V(T_1) \to \RR^{+}\) which is the restriction of \(l\) on \(V(T_1)\), \(l_1 = l|_{V(T_1)}\). Then the labeling \(l_1\) is also non-degenerate and the ultrametric \(d_{l_1}\) is the restriction of the ultrametric \(d_l\) on the set \(V(T_1)\), \(d_{l_1} = d_l|_{V(T_1) \times V(T_1)}\).
\end{lemma}

\begin{proof}
It follows from formula \eqref{e11.3}, Lemma~\ref{l9.14} and the definition of trees and subtrees.
\end{proof}

\begin{theorem}\label{t4.7}
Let \(T = T(l)\) be a labeled tree with non-degenerate labeling. Then the following conditions are equivalent:
\begin{enumerate}
\item \label{t4.7:s1} The ultrametric space \((V(T), d_l)\) is complete.
\item \label{t4.7:s2} For every ray \(R \subseteq T\), 
\begin{equation}\label{t4.7:e2}
V(R) = \{x_1, x_2, \ldots, x_n, x_{n+1}, \ldots\}, \quad E(R) = \{\{x_1, x_2\}, \ldots, \{x_{n}, x_{n+1}\}, \ldots\},
\end{equation}
we have the inequality
\begin{equation}\label{t4.7:e1}
\limsup_{n \to \infty} l(x_n) > 0.
\end{equation}
\end{enumerate}
\end{theorem}

\begin{proof}
\(\ref{t4.7:s1} \Rightarrow \ref{t4.7:s2}\). Let \((V(T), d_l)\) be a complete ultrametric space. We must show that condition~\ref{t4.7:s2} is satisfied. Suppose contrary that there is a ray \(R \subseteq T\) such that \eqref{t4.7:e2} and \(\limsup_{n \to \infty} l(x_n) = 0\) hold. Since all \(l(x_n)\) are nonnegative, the last equality holds if and only if 
\begin{equation}\label{t4.7:e3}
\lim_{n \to \infty} l(x_n) = 0.
\end{equation}
From \eqref{t4.7:e3} and \eqref{e11.3} it follows that
\[
\lim_{n \to \infty} d_l(x_n, x_{n+1}) = \lim_{n \to \infty} \max\{l(x_n), l(x_{n+1})\} = 0,
\]
because \(x_n\) and \(x_{n+1}\) are adjacent in \(R\) and \(R \subseteq T\). Hence, by Proposition~\ref{p4.5}, the sequence \((x_n)_{n \in \mathbb{N}}\) is a Cauchy sequence in the space \((V(T), d_l)\). 

Now, using condition~\ref{t4.7:s1} and Definition~\ref{d2.6}, we can find a point \(v^{*} \in V(T)\) satisfying the equality
\begin{equation}\label{t4.7:e4}
\lim_{n \to \infty} d_l(x_n, v^{*}) = 0.
\end{equation}
If there is \(n_0 \in \NN\) such that \(v^{*} = x_{n_0} \in V(R)\), then, for every \(n \geqslant n_0 + 1\), the path \(P_{x_{n_0}, x_{n}}\) joining \(v^{*}\) and \(x_n\) in \(T\) contains the edge \(\{x_{n_0}, x_{n_0+1}\} \in E(R)\). Since \(l \colon V(T) \to \RR^{+}\) is non-degenerate, \eqref{e11.3} and \(\{x_{n_0}, x_{n_0+1}\} \in E(P_{x_{n_0}, x_{n_0+1}})\) imply 
\[
d_l(v^{*}, x_n) \geqslant \max \{l(x_{n_0}), l(x_{n_0+1})\} > 0
\]
for every \(n \geqslant n_0 + 1\), contrary to \eqref{t4.7:e4}.

Suppose now that \(v^{*} \in V(T) \setminus V(R)\). Then the hull \(H_A\) of the set
\[
A \stackrel{\text{def}}{=} V(R) \cup \{v^{*}\}
\]
is a comb in \(T\) with the tooth \(v^{*}\) (see Definition~\ref{d3.6} and Example~\ref{ex3.7}). Write \(u^{*}\) for the root of \(v^{*}\) in \(H_A\). Since \(u^{*} \neq v^{*}\) and \(u^{*}\) is the only common vertex of \(R\) and of the path \(P_{u^{*}, v^{*}}\) joining \(u^{*}\) and \(v^{*}\) in \(T\), we have 
\[
d_l(x_n, v^{*}) \geqslant d_l(u^{*}, v^{*}) > 0
\]
for all \(x_n \in V(R)\), contrary to~\eqref{t4.7:e4}. Condition~\ref{t4.7:s2} follows.

\(\ref{t4.7:s2} \Rightarrow \ref{t4.7:s1}\). Let \ref{t4.7:s2} hold. We must show that the ultrametric space \((V(T), d_l)\) is complete. According to Definition~\ref{d2.6}, the space is complete if every Cauchy sequence of points of this space is convergent.

Let us consider an arbitrary Cauchy sequence \((y_n)_{n \in \NN} \subseteq V(T)\) and define the range \(A\) of \((y_n)_{n \in \NN}\) as:
\[
(v \in A) \Leftrightarrow (\exists n \in \NN \colon y_n = v).
\]
If \(A\) is finite, then the sequence \((y_n)_{n \in \NN}\) contains an infinite constant subsequence and, consequently, \((y_n)_{n \in \NN}\) is convergent by Proposition~\ref{p2.7}. 

Suppose that \(A\) is infinite and denote by \(H_A\) the hull of \(A\) in \(T\). Let us prove that \(H_A\) is a rayless graph. 

Indeed, let a ray \(R\),
\[
V(R) = \{r_1, r_2, \ldots, r_n, r_{n+1}, \ldots\}, \quad E(R) = \{\{r_1, r_2\}, \ldots, \{r_n, r_{n+1}\}, \ldots\},
\]
be a subgraph of \(H_A\). If the intersection \(A \cap V(R)\) is infinite, then the sequence \((r_n)_{n \in \NN}\) contains a subsequence \((r_{n_k})_{k \in \NN}\) which is Cauchy in \((V(T), d_l)\) and, consequently, in \((V(R), d_l|_{V(R) \times V(R)})\). Using Lemma~\ref{l4.8} and Lemma~\ref{l4.6} we see that the sequence \((r_n)_{n \in \NN}\) is Cauchy in \((V(R), d_l|_{V(R) \times V(R)})\). Now Lemma~\ref{l4.6} implies that \((r_n)_{n \in \NN}\) is a Cauchy sequence in \((V(T), d_l)\). Moreover, the equality
\begin{equation*}
\lim_{n \to \infty} l(r_n) = 0
\end{equation*}
holds by Lemma~\ref{l4.7}. The last equality contradicts \eqref{t4.7:e1} with \((x_n)_{n \in \NN} = (r_n)_{n \in \NN}\). 

Thus, the intersection \(A \cap V(R)\) is finite and \(A\) is infinite. We claim that there is an \emph{infinite} subset \(A^{*}\) of the set \(A \setminus V(R)\) which satisfies the condition:
\begin{enumerate}
\item[\((i^{*})\)] If \(a_1\), \(a_2\) are distinct points of \(A^{*}\), and
\[
A_1 \stackrel{\text{def}}{=} V(R) \cup \{a_1\}, \quad A_2 \stackrel{\text{def}}{=} V(R) \cup \{a_2\},
\]
and \(H_{A_1}\), \(H_{A_2}\) are the corresponding hulls of \(A_1\) and of \(A_2\) in \(T\), then the roots \(r(a_1)\) and \(r(a_2)\) are distinct.
\end{enumerate}

(Recall that the hulls \(H_{A_1}\) and \(H_{A_2}\) in \(T\) are some combs in \(T\), see Example~\ref{ex3.7}). To find a desired \(A^{*} \subseteq A \setminus V(R)\) let us consider a number \(N \in \NN\) such that 
\[
A \cap V(R) \subseteq \{r_1, r_2, \ldots, r_{N}\}
\]
and suppose that, for every \(a \in A \setminus V(R)\), the root \(r(a)\) of the tooth \(a\) in the comb \(H_{V(R) \cup \{a\}}\) belongs to the set \(\{r_1, r_2, \ldots, r_{N}\}\). The graph 
\[
G_{R, N} \stackrel{\text{def}}{=} (r_1, \ldots, r_{N}) \cup \bigcup_{a \in A \setminus V(R)} P_{a, r(a)},
\]
where \((r_1, \ldots, r_{N})\) is the path joining \(r_1\) and \(r_N\) in \(R\) and \(P_{a, r(a)}\) is the path joining \(a\) and \(r(a)\) in the comb \(H_{V(R) \cup \{a\}}\), is a connected subgraph of \(T\) satisfying the conditions
\begin{equation*}
A \subseteq V(G_{R, N}) \quad \text{and} \quad r_n \notin V(G_{R, N}) 
\end{equation*}
for every \(n \geqslant N+1\). Since \(r_n \in V(H_A)\) holds for every \(n \in \NN\), the inclusion
\[
V(H_A) \subseteq V(G_{R, N})
\]
is false, contrary to Proposition~\ref{p3.7}. Hence, there is an element \(y_{n_1}\) of the sequence \((y_{n})_{n \in \NN}\) such that \(y_{n_1} \in A \setminus V(R)\) and \(r(y_{n_1}) = r_{N_1}\) and \(N_1 > N\). If for every \(a \in A \setminus V(R)\) the root \(r(a)\) belongs to \(\{r_1, \ldots, r_N, \ldots, r_{N_1}\}\), then repeating the above procedure with the graph \(G_{R, N_1}\) gives us \(y_{n_2} \in A \setminus V(R)\) with \(r(y_{n_2}) = r_{N_2}\) such that \(n_2 > n_1\) and \(N_2 > N_1\) and so on. 

Let us consider the sequence \((y_{n_k})_{k \in \NN}\), whose elements are inductively defined above, and write
\[
A^{*} \stackrel{\text{def}}{=} \{y_{n_k} \colon k \in \NN\}.
\]
Then condition~\((i^{*})\) satisfies with this \(A^{*}\). Since \((y_{n})_{n \in \NN}\) is a Cauchy sequence in \((V(T), d_l)\), the sequence \((y_{n_k})_{k \in \NN}\) is also Cauchy. It is easy to prove that the inequality
\begin{equation}\label{t4.7:e12}
d_l(r(y_{n_{k_1}}), r(y_{n_{k_2}})) \leqslant d_l(y_{n_{k_1}}, y_{n_{k_2}})
\end{equation}
holds for all \(k_1\), \(k_2 \in \NN\) (see Figure~\ref{fig3}). Consequently, \((r(y_{n_k}))_{k \in \NN}\) is a Cauchy sequence in \((V(T), d_l)\).

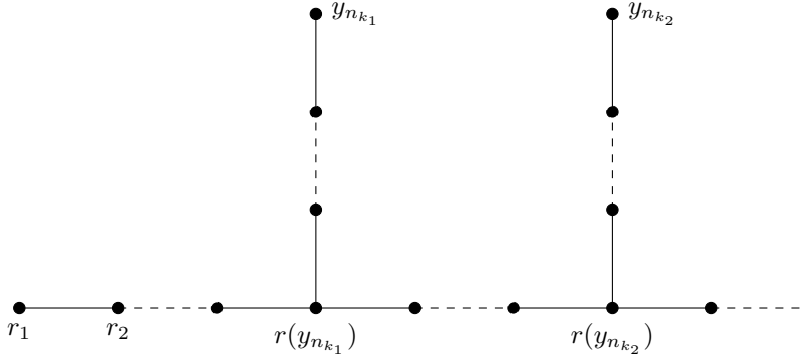
\begin{figure}[ht]
\begin{center}
\begin{tikzpicture}[
level 1/.style={level distance=\levdist,sibling distance=24mm},
level 2/.style={level distance=\levdist,sibling distance=12mm},
level 3/.style={level distance=\levdist,sibling distance=6mm},
solid node/.style={circle,draw,inner sep=1.5,fill=black},
micro node/.style={circle,draw,inner sep=0.5,fill=black}]

\def\dx{1.3cm}
\draw (\dx, 0) node [solid node, label= below:\(r_1\)] {} -- (2*\dx, 0) node [solid node, label= below:\(r_2\)] {};
\draw [dashed] (2*\dx, 0) -- (3*\dx, 0) node [solid node] {};
\draw (3*\dx, 0) -- (4*\dx, 0) node [solid node, label= below:\(r(y_{n_{k_1}})\)] {} -- (5*\dx, 0) node [solid node] {};
\draw (4*\dx, 0) -- (4*\dx, \dx) node [solid node] {};
\draw [dashed] (4*\dx, \dx) -- (4*\dx, 2*\dx) node [solid node] {};
\draw (4*\dx, 2*\dx) -- (4*\dx, 3*\dx) node [solid node, label= right:\(y_{n_{k_1}}\)] {};

\draw [dashed] (5*\dx, 0) -- (6*\dx, 0) node [solid node] {};
\draw (6*\dx, 0) -- (7*\dx, 0) node [solid node, label= below:\(r(y_{n_{k_2}})\)] {} -- (8*\dx, 0) node [solid node] {};
\draw (7*\dx, 0) -- (7*\dx, \dx) node [solid node] {};
\draw [dashed] (7*\dx, \dx) -- (7*\dx, 2*\dx) node [solid node] {};
\draw (7*\dx, 2*\dx) -- (7*\dx, 3*\dx) node [solid node, label= right:\(y_{n_{k_2}}\)] {};

\draw [dashed] (8*\dx, 0) -- (9*\dx, 0);
\end{tikzpicture}
\end{center}
\caption{The path \(\bigl(r(y_{n_{k_1}}), \ldots, r(y_{n_{k_2}})\bigr)\) is a subgraph of the path \(\bigl(y_{n_{k_1}}, \ldots, r(y_{n_{k_1}}), \ldots, r(y_{n_{k_2}}), \ldots, y_{n_{k_2}}\bigr)\). It implies inequality~\eqref{t4.7:e12}.}
\label{fig3}
\end{figure}

Now using Lemma~\ref{l4.8} and Lemma~\ref{l4.6} we can prove that the sequence \((r_n)_{n \in \NN}\) of all vertices of the ray \(R\) is also a Cauchy sequence in \((V(T), d_l)\). Hence, by Lemma~\ref{l4.7}, we have the equality
\[
\lim_{n \to \infty} l(r_n) = 0,
\]
that contradicts condition \ref{t4.7:s2}. 

\emph{Thus, the hull \(H_A\) is rayless}. Since \(A\) is infinite, \(H_A\) has a vertex \(v^{*}\) of infinite degree by Proposition~\ref{p3.2}. To complete the proof it suffices to show that \((y_n)_{n \in \NN}\) converges to the point \(v^{*}\) in \((V(T), d_l)\),
\begin{equation*}
\lim_{n \to \infty} d_l(y_n, v^{*}) = 0.
\end{equation*}

Let us consider the subgraph \(F_{v^{*}}\) obtained from \(H_A\) by deleting the vertex \(v^{*}\),
\[
V(F_{v^{*}}) = V(H_A) \setminus \{v^{*}\}, \quad E(F_{v^{*}}) = \{\{u, v\} \in E(H_A) \colon u \neq v^{*} \neq v\}.
\]
Since \(H_A\) is a tree (as a subtree of \(T\)), \(F_{v^{*}}\) is a forest and, for every connected component \(T'\) of \(F_{v^{*}}\), there is a unique \(p\in V(N(v^{*}))\) such that
\begin{equation}\label{t4.7:e11}
p \in V(T'),
\end{equation}
where \(N(v^{*})\) is the neighborhood of \(v^{*}\) in \(H_A\). We will write \(T^{p}\) for the component \(T'\) if \eqref{t4.7:e11} holds with \(p\in V(N(v^{*}))\).

It is clear that
\begin{equation}\label{t4.7:e13}
H_A = \left(\bigcup_{p \in V(N(v^{*}))} T^{p}\right) \cup S(v^{*})
\end{equation}
holds, where \(S(v^{*})\) is the star with the center \(v^{*}\) and the vertex set 
\[
V(S(v^{*})) = V(N(v^{*})).
\]

We claim that the set \(V(T^p) \cap A\) is nonempty for every \(p \in V(N(v^{*}))\). Indeed, suppose that there is \(p^{*} \in V(N(v^{*}))\) such that
\begin{equation}\label{t4.7:e14}
V(T^{p^{*}}) \cap A = \varnothing.
\end{equation}
Let us denote by \(S_{p^{*}}\) the graph which is obtained from \(S(v^{*})\) by deleting of the vertex \(p^{*}\), i.e., 
\[
V(S_{p^{*}}) = V(S(v^{*})) \setminus \{p^{*}\}
\]
holds and
\[
\bigl(\{x, y\} \in E(S_{p^{*}})\bigr) \Leftrightarrow \bigl(\{x, y\} \in E(S(v^{*})) \text{ and } x \neq p^{*} \neq y\bigr)
\]
is valid for all \(x\), \(y \in V(N(v^{*}))\). Then \(S_{p^{*}}\) is a star with the center \(v^{*}\). From \eqref{t4.7:e14} it follows that the union
\[
\left(\bigcup_{\substack{p \in V(N(v^{*}))\\ p \neq p^{*}}} T^{p} \right) \cup S_{p^{*}}
\]
is a subtree of \(T\) and the set \(A\) is a subset of the vertex set of this subtree. Hence, by Definition~\ref{d3.6}, we have
\[
p^{*} \notin V(H_A),
\]
contrary to \eqref{t4.7:e13}. Using the conditions 
\[
\delta_{H_A} (v^{*}) = \infty \quad \text{and} \quad V(T^{p}) \cap A \neq \varnothing
\]
for every \(p \in V(N(v^{*}))\), we can find a subsequence \((y_{n_m})_{m \in \NN}\) of the sequence \((y_{n})_{n \in \NN}\) such that for every \(m \in \NN\) there is \(p(m) \in V(N(v^{*}))\) satisfying \(y_{n_m} \in T^{p}\), and \(p(m_1) \neq p(m_2)\) whenever \(m_1 \neq m_2\). Then for every pair of distinct \(m_1\), \(m_2 \in \NN\) the path \(P_{y_{n_{m_1}}, y_{n_{m_2}}}\) joining \(y_{n_{m_1}}\) and \(y_{n_{m_2}}\) in \(H_A\) contains the vertex \(v^{*}\). Hence, by definition~\eqref{e11.3}, we have the inequality
\begin{equation}\label{t4.7:e15}
d_l\left(y_{n_{m_1}}, y_{n_{m_2}}\right) \geqslant \max \left\{d_l\left(y_{n_{m_1}}, v^{*}\right), d_l\left(v^{*}, y_{n_{m_2}}\right)\right\}
\end{equation}
whenever \(m_1\), \(m_2 \in \NN\). By Proposition~\ref{p2.7}, the sequence \((y_{n})_{n \in \NN}\) is convergent if \((y_{n_m})_{m \in \NN}\) is convergent. Using Proposition~\ref{p4.5} and inequality \eqref{t4.7:e15} with \(n_{m_1} = n_{m}\) and \(n_{m_2} = n_{m+1}\) we obtain
\[
0 = \lim_{m \to \infty} d_l\left(y_{n_{m}}, y_{n_{m+1}}\right) \geqslant \limsup_{m \to \infty} d_l\left(y_{n_{m}}, v^{*}\right),
\]
that implies 
\[
\lim_{m \to \infty} d_l\left(y_{n_{m}}, v^{*}\right) = 0.
\]
Thus, \((y_{n_m})_{m \in \NN}\) is convergent to \(v^{*}\).
\end{proof}

Condition~\ref{t4.7:s2} of Theorem~\ref{t4.7} is vacuously true for every rayless tree \(T\). Moreover, if \(R \subseteq T\) is a ray satisfying \eqref{t4.7:e2}, then it is easy to construct a non-degenerate labeling \(l \colon V(T) \to \RR^{+}\) such that \eqref{t4.7:e1} does not hold. Thus, Theorem~\ref{t4.7} implies the next corollary.

\begin{corollary}\label{c4.7}
The following statements are equivalent for every tree \(T\):
\begin{enumerate}
\item \label{c4.7:s1} The ultrametric space \((V(T), d_{l})\) is complete for every non-degenerate labeling \(l \colon V(T) \to \RR^{+}\).
\item \label{c4.7:s2} \(T\) is rayless.
\end{enumerate}
\end{corollary}

Recall that a metric space \((X, d)\) is \emph{discrete} if every point of \(X\) is isolated.

\begin{theorem}\label{t4.11}
Let \(T = T(l)\) be a labeled tree with non-degenerate labeling. Then the following statements are equivalent:
\begin{enumerate}
\item \label{t4.11:s1} The ultrametric space \((V(T), d_{l})\) is discrete.
\item \label{t4.11:s2} For every \(v^{*} \in V(T)\) we have either \(l(v^{*}) > 0\) or \(l(v^{*}) = 0\) and 
\begin{equation}\label{t4.11:e1}
0 < \inf_{u \in V(N(v^{*}))} l(u),
\end{equation}
where \(N(v^{*})\) is the neighborhood of \(v^{*}\) in \(T\).
\end{enumerate}
\end{theorem}

\begin{proof}
\(\ref{t4.11:s1} \Rightarrow \ref{t4.11:s2}\). Let \((V(T), d_{l})\) be a discrete ultrametric space. If \ref{t4.11:s2} does not hold, then there is a vertex \(v^{*}\) such that \(l(v^{*}) = 0\) and 
\[
\inf_{u \in V(N(v^{*}))} l(u) = 0.
\]
Hence, there exists a sequence \((u_n)_{n \in \NN} \subseteq V(N(v^{*}))\) such that
\[
\lim_{n \to \infty} l(u_n) = 0.
\]
The last limit relation, the equality \(l(v^{*}) = 0\), equality \eqref{e11.3} and the definition of the neighborhoods of vertices of graph imply that
\begin{equation}\label{t4.11:e0}
\lim_{n \to \infty} d(v^{*}, u_n) = 0
\end{equation}
holds. Hence, \(v^{*}\) is an accumulation point in \((V(T), d_{l})\), contrary to statement \ref{t4.11:s1}. 

\(\ref{t4.11:s2} \Rightarrow \ref{t4.11:s1}\). Let \ref{t4.11:s2} hold. Statement \ref{t4.11:s1} is valid if \(E(T) = \varnothing\). Indeed, in this case the vertex set of \(T\) is a single-point set \(\{v^{*}\}\). Thus, \(V(N(v^{*})) = \varnothing\) holds and, consequently, we have
\[
\inf_{u \in V(N(v^{*}))} l(u) = \inf_{u \in \varnothing} l(u) = +\infty,
\]
that implies \eqref{t4.11:e1}. (We consider here the empty set \(\varnothing\) as a subset of \([-\infty, \infty]\) and adopt the standard agreement on the supremum and infimum of empty set.)

Let \(E(T) \neq \varnothing\) hold. 

Suppose that \(v^{*}\) is a vertex of \(T\) such that \(l(v^{*}) > 0\). Then \eqref{e11.3} implies
\[
d_l(u, v^{*}) \geqslant l(v^{*})
\]
for every \(u \in V(T) \setminus \{v^{*}\}\). Hence, \(v^{*}\) is an isolated point of \((V(T), d_{l})\).

Let us consider now the case when \(v^{*} \in V(T)\) has the zero label, \(l(v^{*}) = 0\), and assume that we have \(\delta_{T}(v^{*}) < \infty\). Since the inequality \(\max\{l(u), l(v^{*})\} > 0\) holds for every \(u \in V(N(v^{*}))\), from \(0 < \delta_{T}(v^{*}) < \infty\) follows the inequality
\begin{equation}\label{t4.11:e3}
\min_{u \in V(N(v^{*}))} l(u) > 0. 
\end{equation}
Let \(u_0 \in V(T) \setminus \{v^{*}\}\). Then there is \(u^{*} \in V(P_{v^{*}, u_0})\) such that
\begin{equation}\label{t4.11:e2}
u^{*} \in V(N(v^{*})).
\end{equation}
Now from \eqref{e11.3} and \eqref{t4.11:e2} it follows that
\[
d_l(v^{*}, u_0) = \max_{u \in V(P_{v^{*}, u_0})} l(u) \geqslant d_l(v^{*}, u^{*}) \geqslant \min_{u \in V(N(v^{*}))} l(u) > 0.
\]
Hence, \(v^{*}\) is an isolated point of \((V(T), d_{l})\). 

Using inequality \eqref{t4.11:e1} instead of \eqref{t4.11:e3} and repeating the above arguments, we obtain that \(v^{*}\) is also isolated for the case \(l(v^{*}) = 0\) and \(\delta_{T}(v^{*}) = \infty\). Thus, the ultrametric space \((V(T), d_{l})\) is discrete. 
\end{proof}

\begin{corollary}\label{c4.12}
The following statements are equivalent for every tree \(T\):
\begin{enumerate}
\item \label{c4.12:s1} The ultrametric space \((V(T), d_{l})\) is discrete for every non-degenerate labeling \(l \colon V(T) \to \RR^{+}\).
\item \label{c4.12:s2} The tree \(T\) is locally finite.
\end{enumerate}
\end{corollary}

\begin{proposition}\label{p4.13}
Let \(T\) be a tree. Then the ultrametric space \((V(T), d_{l})\) contains a dense discrete subset for every non-degenerate \(l \colon V(T) \to \RR^{+}\).
\end{proposition}

\begin{proof}
Let \(l \colon V(T) \to \RR^{+}\) be non-degenerate. It was shown in the second part of the proof of Theorem~\ref{t4.11} that \(v \in V(T)\) is an isolated point of the ultrametric space \((V(T), d_l)\) if at least one of the conditions:
\begin{itemize}
\item \(\delta(v) < \infty\),
\item \(l(v) > 0\),
\item \(\delta(v) = \infty\), \(l(v) = 0\) and \(\inf_{u \in V(N(v^{*}))} l(u) > 0\)
\end{itemize}
is valid. Arguing as in the first part of the proof of Theorem~\ref{t4.11}, we can show that, for every \(v\) satisfying conditions \(\delta(v) = \infty\) and 
\[
l(v) = 0 = \inf_{u \in V(N(v^{*}))} l(u),
\]
there is a sequence \((u_n)_{n \in \NN} \subseteq V(N(v))\) which converges to \(v\) (see~\eqref{t4.11:e0}). Now it suffices to note that all elements of this sequence are isolated points of \((V(T), d_l)\) because \(l(v) = 0\) holds and \(l\) is non-degenerate.
\end{proof}

The following result gives us the necessary and sufficient conditions under which the ultrametric space \((V(T), d_l)\) is totally bounded.

\begin{theorem}\label{t11.8}
Let \(T = T(l)\) be a labeled tree with non-degenerate labeling. Then the following statements are equivalent:
\begin{enumerate}
\item \label{t11.8:s1} The ultrametric space \((V(T), d_{l})\) is totally bounded.
\item \label{t11.8:s2} The set 
\begin{equation}\label{t11.8:e1}
V_{\varepsilon} = \bigl\{v \in V(T) \colon l(v) \geqslant \varepsilon\bigr\}
\end{equation}
is finite for every \(\varepsilon > 0\) and the inequality \(\delta_{T}(v) < \infty\) holds whenever \(l(v) > 0\).
\end{enumerate}
\end{theorem}

\begin{proof}
\(\ref{t11.8:s1} \Rightarrow \ref{t11.8:s2}\). Let \ref{t11.8:s1} hold. Suppose that the set \(V_{\varepsilon}\) is infinite for some \(\varepsilon > 0\). Using the definition of \(d_{l}\) (see~\eqref{e11.3}) it is easy to prove the inequality
\begin{equation*}
d_{l}(v_1, v_2) \geqslant \varepsilon
\end{equation*}
for all distinct \(v_1\), \(v_2 \in V_{\varepsilon}\). Hence, the subspace \((V_{\varepsilon}, d_{l}|_{V_{\varepsilon} \times V_{\varepsilon}})\) of totally bounded ultrametric space \((V(T), d_{l})\) is not totally bounded, contrary to Corollary~\ref{c2.17}. Thus, \(V_{\varepsilon}\) is finite for every \(\varepsilon > 0\).

Assume now that \(T\) contains a vertex \(v^*\) of infinite degree, \(\delta (v^*) = \infty\), and \(l(v^{*}) > 0\) holds.

Let \(N(v^{*})\) be the neighborhood of \(v^{*}\). The equality \(\delta_{T}(v^*) = \infty\) implies that \(V(N(v^{*}))\) has an infinite cardinality. For all distinct \(u_1\), \(u_2 \in V(N(v^{*}))\) the unique path joining \(u_1\) and \(u_2\) in \(T\) has the form \((u_1, v^{*}, u_2)\). Hence, 
\[
d_l(u_1, u_2) \geqslant l(v^{*}) > 0
\]
holds by \eqref{e11.3}. It implies that the ultrametric space \((V(N(v^{*})), d_l|_{V(N(v^{*})) \times V(N(v^{*}))})\) is not totally bounded, contrary to \ref{t11.8:s1}. 

\(\ref{t11.8:s2} \Rightarrow \ref{t11.8:s1}\). Let \ref{t11.8:s2} hold. We must show that \((V(T), d_{l})\) is totally bounded. It is clear that \((V(T), d_{l})\) is totally bounded for finite \(T\). Let us consider the case when \(T\) is infinite.

By Proposition~\ref{p2.11}, the ultrametric space \((V(T), d_{l})\) is totally bounded if every sequence of vertices of \(T\) contains a Cauchy subsequence. Let us consider a sequence \((u_j^0)_{j \in \mathbb{N}}\) of pairwise distinct points of \(V(T)\). Let \((\varepsilon_i)_{i \in \mathbb{N}}\) be a decreasing sequence of strictly positive real numbers such that \(\lim_{i \to \infty} \varepsilon_i = 0\). Statement~\ref{t11.8:s2} implies that the set \(V_{\varepsilon_{1}}\) is finite. Write \(G^{1}\) for the subgraph of \(T\) induced by \(V(T) \setminus V_{\varepsilon_{1}}\), i.e., \(V(G^{1}) = V(T) \setminus V_{\varepsilon_{1}}\) and
\[
\bigl(\{u, v\} \in E(G^{1})\bigr) \Leftrightarrow \bigl(u, v \in V(G^{1}) \text{ and } \{u, v\} \in E(T)\bigr).
\]
Every connected component of \(G^{1}\) is a tree. Since \(V_{\varepsilon_{1}}\) is a finite set, and \(\delta_{T}(v) < \infty\) holds for every \(v \in V_{\varepsilon_{1}}\), the number of these components are finite. Since \(T\) is an infinite tree, there is an infinite subtree \(T^{1}\) of \(T\) with \(l(v^{1}) < \varepsilon_{1}\) for all \(v^{1} \in V(T^{1})\) and such that \((u_{j}^{1})_{j \in \mathbb{N}} \subseteq V(T^{1})\) holds for an infinite subsequence \((u_{j}^{1})_{j \in \mathbb{N}}\) of the sequence \((u_{j}^{0})_{j \in \mathbb{N}} \subseteq V(T)\). Write \(u^{1} = u_1^1\). 

Let us consider the subgraph \(G^{2}\) of \(T^{1}\) induced by \(V(T^{1}) \setminus V_{\varepsilon_{2}}\). As above, the finiteness of \(V_{\varepsilon_{2}}\) implies the existence of an infinite tree \(T^{2} \subseteq T^{1}\) and a subsequence \((u_{j}^{2})_{j \in \mathbb{N}}\) of the sequence \((u_{j}^{1})_{j \in \mathbb{N}} \subseteq V(T^{1})\) for which \((u_{j}^{2})_{j \in \mathbb{N}} \subseteq V(T^{2})\) holds. Let us write \(u^{2} = u_1^2\). 

By induction, for every \(i \geqslant 2\), we find an infinite connected component \(T^{i+1}\) of the subgraph \(G^{i+1}\) of \(T^{i}\) induced by \(V(T^{i}) \setminus V_{\varepsilon_{i+1}}\) and a subsequence \((u_{j}^{i+1})_{j \in \mathbb{N}}\) of the sequence \((u_{j}^{i})_{j \in \mathbb{N}}\) such that
\begin{equation}\label{t11.8:e8}
(u_{j}^{i+1})_{j \in \mathbb{N}} \subseteq V(T^{i+1}).
\end{equation}
Write \(u^{i+1}\) for the first element \(u_{1}^{i+1}\) of this sequence.

Let us consider now the sequence \((u^{i})_{i \in \mathbb{N}}\). It is clear that \((u^{i})_{i \in \mathbb{N}}\) is a subsequence of the original sequence \((u_j^0)_{j \in \mathbb{N}}\). From~\eqref{t11.8:e8} it follows that
\[
l(u^{i+1}) < \varepsilon_{i+1}
\]
holds for every \(i \in \mathbb{N}\). The last inequality and the limit relation \(\lim_{i \to \infty} \varepsilon_i = 0\) imply
\begin{equation}\label{t11.8:e9}
\lim_{i \to \infty} l(u^{i}) = 0.
\end{equation}
Moreover, since for every \(i \geqslant 2\) the points \(u^{i}\) and \(u^{i+1}\) are vertices of the tree \(T^{i}\) and the inequality \(l(v) \leqslant \varepsilon_{i}\) holds for every \(v \in V(T^{i})\), formula \eqref{e11.3} implies the inequality
\[
d_{l}(u^{i}, u^{i+1}) \leqslant l(u^{i-1})
\]
for every \(i \geqslant 2\). Now using Proposition~\ref{p4.5} and limit relation~\eqref{t11.8:e9} we obtain that \((u^{i})_{i \in \mathbb{N}}\) is a Cauchy sequence.

The same reasons show that \((u_j^0)_{j \in \mathbb{N}}\) contains a Cauchy subsequence whenever \((u_j^0)_{j \in \mathbb{N}}\) contains an infinite subsequence of pairwise distinct members.

To complete the proof, we note only that if all subsequences of pairwise distinct members of a sequence are finite, then there is an infinite constant subsequence of that sequence and this constant subsequence obviously is a Cauchy sequence. 
\end{proof}

\begin{corollary}\label{c4.8}
Let \(T = T(l)\) be a labeled tree with non-degenerate labeling. If the ultrametric space \((V(T), d_l)\) is totally bounded, then the set \(V(T)\) is at most countable.
\end{corollary}

\begin{proof}
It suffices to show that the inequality 
\begin{equation}\label{c4.8:e1}
\delta_{T}(v^*) \leqslant \aleph_0
\end{equation}
holds for every vertex \(v^*\) of \(T\), where \(\aleph_0\) is the cardinality of \(\mathbb{N}\).

Let \((V(T), d_l)\) be totally bounded and let \(v^*\) be a vertex of \(T\). Inequality \eqref{c4.8:e1} follows directly from Theorem~\ref{t11.8} if \(l(v^*) > 0\). Suppose that \(l(v^*) = 0\) holds. Then for every \(\{u, v^*\} \in E(T)\) we have the inequality \(l(u) > 0\). Consequently, the vertex set \(V(N(v^{*}))\) satisfies the inclusion 
\begin{equation}\label{c4.8:e2}
V(N(v^{*})) \subseteq \bigcup_{n \in \mathbb{N}} V_{1/n},
\end{equation}
where \(V_{1/n}\) is defined by \eqref{t11.8:e1} with \(\varepsilon = 1/n\). By Theorem~\ref{t11.8}, \(V_{1/n}\) is finite for every \(n \in \mathbb{N}\). Hence, \(\bigcup_{n \in \mathbb{N}} V_{1/n}\) is at most countable. Now inequality \eqref{c4.8:e1} follows from \eqref{c4.8:e2}.
\end{proof}

\begin{theorem}\label{t11.16}
Let \(T = T(l)\) be a labeled tree with non-degenerate labeling. Then the following conditions are equivalent:
\begin{enumerate}
\item \label{t11.16:s1} The ultrametric space \((V(T), d_{l})\) is discrete and totally bounded.
\item \label{t11.16:s2} The tree \(T\) is locally finite and the set
\begin{equation*}
V_{\varepsilon} = \bigl\{v \in V(T) \colon l(v) \geqslant \varepsilon\bigr\}
\end{equation*}
is finite for every \(\varepsilon > 0\).
\end{enumerate}
\end{theorem}

\begin{proof}
\(\ref{t11.16:s1} \Rightarrow \ref{t11.16:s2}\). Let \ref{t11.16:s1} hold. Then the set \(V_{\varepsilon}\) is finite for every \(\varepsilon > 0\) by Theorem~\ref{t11.8}. 

Assume now that \(T\) contains a vertex \(v^*\) of infinite degree, \(\delta (v^*) = \infty\). Then, using Theorem~\ref{t11.8} again, we obtain the equality
\begin{equation}\label{t11.16:e4}
l(v^{*}) = 0.
\end{equation}
By Theorem~\ref{t4.11}, equality~\eqref{t11.16:e4} and discreteness of \((V(T), d_{l})\) imply that there is \(\varepsilon^{*} > 0\) such that
\[
\inf_{u \in V(N(v^{*}))} l(u) \geqslant \varepsilon^{*}.
\]
Hence, we have the inclusion \(V(N(v^{*})) \subseteq V_{\varepsilon^{*}}\). It was shown above that \(V_{\varepsilon}\) is finite for every \(\varepsilon > 0\). Consequently, \(V(N(v^{*}))\) is also finite as a subset of a finite set, contrary to \(\delta_T (v^{*}) = \infty\).

\(\ref{t11.16:s2} \Rightarrow \ref{t11.16:s1}\). Let \ref{t11.16:s2} hold. Then \(T\) is locally finite and, consequently, the ultrametric space \((V(T), d_{l})\) is discrete by Corollary~\ref{c4.12}. To complete the proof, it suffices to note that \((V(T), d_{l})\) is totally bounded by Theorem~\ref{t11.8}.
\end{proof}

Corollary~\ref{c4.12} and Theorem~\ref{t11.16} imply the following.

\begin{corollary}\label{c4.15}
Let \(T\) be a tree. Then the following statements are equivalent:
\begin{enumerate}
\item \label{c4.15:s1} There is a non-degenerate labeling \(l_1 \colon V(T) \to \RR^{+}\) for which the ultrametric space \((V(T), d_{l_1})\) is discrete and totally bounded.
\item \label{c4.15:s2} The ultrametric space \((V(T), d_{l})\) is discrete for every non-degenerate labeling \(l \colon V(T) \to \RR^{+}\).
\end{enumerate}
\end{corollary}

\begin{proof}
If \ref{c4.15:s1} holds, then \(T\) is locally finite by Theorem~\ref{t11.16}, that implies \ref{c4.15:s2} by Corollary~\ref{c4.12}.

Conversely, suppose that \ref{c4.15:s2} holds. Then, using Corollary~\ref{c4.12} again, we see that \(T\) is locally finite. If \(T\) is finite, then \ref{c4.15:s1} is trivially valid. Every infinite locally finite tree has countable vertex set. Thus, all vertices of \(T\) can be numbered in a sequence \((v_n)_{n \in \NN}\) of pairwise different points and we can define a labeling \(l_1 \colon V(T) \to \RR^{+}\) as
\[
l_1(v_n) = \frac{1}{n}
\]
for every \(n \in \NN\). Then \(l_1\) is a non-degenerate labeling and \(T(l_1)\) satisfies condition \ref{t11.16:s2} of Theorem~\ref{t11.16}. Hence, \((V(T), d_{l_1})\) is discrete and totally bounded.
\end{proof}

Using Corollaries~\ref{c4.12} and \ref{c4.15} we obtain.

\begin{corollary}\label{c4.18}
Let \(T\) be a tree. Then the following statements are equivalent:
\begin{enumerate}
\item\label{c4.18:s1} \(T\) is locally finite.
\item\label{c4.18:s2} There is a non-degenerate labeling \(l_1 \colon V(T) \to \RR^{+}\) for which \((V(T), d_{l_1})\) is discrete and totally bounded.
\end{enumerate}
\end{corollary}

\begin{theorem}\label{t4.18}
Let \(T = T(l)\) be a labeled tree with non-degenerate labeling. Then the following statements are equivalent:
\begin{enumerate}
\item \label{t4.18:s1} The ultrametric space \((V(T), d_{l})\) is compact.
\item \label{t4.18:s2} The tree \(T\) is rayless and the set
\begin{equation*}
V_{\varepsilon} = \bigl\{v \in V(T) \colon l(v) \geqslant \varepsilon\bigr\}
\end{equation*}
is finite for every \(\varepsilon > 0\) and the inequality \(\delta_{T}(v) < \infty\) holds whenever \(l(v) > 0\).
\end{enumerate}
\end{theorem}

\begin{proof}
\(\ref{t4.18:s1} \Rightarrow \ref{t4.18:s2}\). Let \((V(T), d_{l})\) be a compact ultrametric space. Every compact metric space is totally bounded and complete by Corollary~\ref{c2.9}. Hence, by Theorem~\ref{t11.8}, for every \(\varepsilon > 0\) the set \(V_{\varepsilon}\) is finite, and \(\delta_{T}(v) < \infty\) holds for all vertices \(v\) with \(l(v) > 0\). 

Suppose that there is a ray \(R \subseteq T\). Then the completeness of \((V(T), d_{l})\) and Theorem~\ref{t4.7} imply the existence of \(\varepsilon^{*} > 0\) and of a sequence \((x_n)_{n \in \NN}\) of pairwise distinct vertices of \(R\) such that
\[
\limsup_{n \to \infty} l(x_n) \geqslant \varepsilon^{*} > 0.
\]
Thus, the set \(\{v \in V(R) \colon l(v) \geqslant \frac{1}{2} \varepsilon^{*}\}\) is infinite, contrary to the finiteness of the set \(V_{\frac{\varepsilon^{*}}{2}}\) which contains \(\{v \in V(R) \colon l(v) \geqslant \frac{1}{2} \varepsilon^{*}\}\).

\(\ref{t4.18:s2} \Rightarrow \ref{t4.18:s1}\). Let \ref{t4.18:s2} hold. Then \ref{t4.18:s1} follows from the Spatial Criterion (Corollary~\ref{c2.9}) and Theorems~\ref{t4.7}, \ref{t11.8}.
\end{proof}

\begin{theorem}\label{t7}
Let \(T\) be a tree. Then the following statements are equivalent:
\begin{enumerate}
\item \label{t7:s1} \(T\) is rayless, and the set \(V(T)\) is at most countable, and, for every \(\{x, y\} \in E(T)\), at least one from the degrees \(\delta(x)\) and \(\delta(y)\) is finite.
\item \label{t7:s2} There is a non-degenerate labeling \(l_1 \colon V(T) \to \mathbb{R}^{+}\) for which the ultrametric space \((V(T), d_{l_1})\) is compact.
\end{enumerate}
\end{theorem}

\begin{proof}
\(\ref{t7:s1} \Rightarrow \ref{t7:s2}\). Let \ref{t7:s1} hold. Let us define the subsets \(V^F\) and \(V^I\) of \(V(T)\) as
\begin{gather*}
V^F \stackrel{\text{def}}{=} \bigl\{v \in V(T) \colon \delta(v) \text{ is finite}\bigr\}, \quad
V^I \stackrel{\text{def}}{=} \bigl\{v \in V(T) \colon \delta(v) \text{ is infinite}\bigr\}.
\end{gather*}
The set \(V(T)\) is at most countable. Consequently, there is a labeling \(l_1 \colon V(T) \to \mathbb{R}^{+}\) such that:
\begin{itemize}
\item the set \(V_{\varepsilon} = \{v \in V(T) \colon l_1(v)  \geqslant \varepsilon\}\) is finite for every \(\varepsilon > 0\); 
\item the inequality \(l_1(u) > 0\) holds for every \(u \in V^F\); 
\item the equality \(l_1(w) = 0\) holds for every \(w \in V^I\). 
\end{itemize}

These properties of \(l_1\) and statement \ref{t7:s1} imply the inequality
\[
\max \bigl\{l_1(x), l_1(y)\bigr\} > 0
\]
for every \(\{x, y\} \in E(T)\). Hence, \(l_1 \colon V(T) \to \mathbb{R}^{+}\) is a non-degenerate labeling. Thus, \((V(T), d_{l_1})\) is compact by Theorem~\ref{t4.18}.

\(\ref{t7:s2} \Rightarrow \ref{t7:s1}\). Let \(l_1 \colon V(T) \to \mathbb{R}^{+}\) be a non-degenerate labeling for which the ultrametric space \((V(T), d_{l_1})\) is compact. Then, from Theorem~\ref{t4.18} it follows that \(T\) is rayless. Moreover, since every compact metric space is totally bounded, the set \(V(T)\) is at most countable by Corollary~\ref{c4.8}.

To complete the proof it is enough to show that the inequality
\begin{equation*}
\min\bigl\{\delta(u), \delta(v)\bigr\} < \infty
\end{equation*}
holds for every \(\{u, v\} \in E(T)\). Assume to the contrary that there exists \(\{u, v\} \in E(T)\) such that
\[
\delta(u) = \delta(v) = \aleph_0.
\]
Since \(l_1 \colon V(T) \to \mathbb{R}^{+}\) is a non-degenerate labeling, \(\{u, v\} \in E(T)\) implies
\[
\max\bigl\{l_1(u), l_1(v)\bigr\} > 0.
\]
Without loss of generality, we may assume that \(l_1(v) > 0\). The last inequality, the inequality \(\delta(v) > 0\) and Theorem~\ref{t4.18} imply that \((V(T), d_{l_1})\) is not compact, contrary to the definition of \(l_1 \colon V(T) \to \mathbb{R}^{+}\).
\end{proof}

Corollary~\ref{c4.7} and Theorem~\ref{t7} imply the following.

\begin{corollary}\label{c4.23}
Let \(T\) be a tree. If there is a non-degenerate labeling \(l_1 \colon V(T) \to \RR^{+}\) for which the ultrametric space \((V(T), d_{l_1})\) is compact, then \((V(T), d_{l})\) is complete for every non-degenerate labeling \(l \colon V(T) \to \RR^{+}\).
\end{corollary}

\begin{remark}\label{r4.21}
It is interesting to compare Corollary~\ref{c4.23} with the following theorem: ``A metrizable topological space \((X, \tau)\) is compact if and only if every metric generated the topology \(\tau\) is complete.'' This result was proved by Niemytzki and Tychonoff in 1928 \cite{NTFM1928}. There is also an ultrametric modification of Niemytzki---Tychonoff theorem recently obtained by Yoshito Ishiki~\cite{Isha2020}.
\end{remark}

The next corollary follows from Proposition~\ref{p3.2} and Theorems~\ref{t11.16}, \ref{t4.18}.

\begin{corollary}\label{c4.20}
The following statements are equivalent for every tree \(T\):
\begin{enumerate}
\item \label{c4.20:s1} There are non-degenerate labelings \(l_1 \colon V(T) \to \RR^{+}\) and \(l_2 \colon V(T) \to \RR^{+}\) such that \((V(T), d_{l_1})\) is a compact ultrametric space and \((V(T), d_{l_2})\) is a discrete totally bounded ultrametric space.
\item \label{c4.20:s2} \(T\) is a finite tree.
\end{enumerate}
\end{corollary}

\section{Some examples and conjectures}

Let us consider examples of totally bounded non-complete ultrametric spaces, and compact ultrametric spaces generated by labeled trees having infinitely many vertices of infinite degree. To construct these examples, we use the \emph{gluing} a given set of labeled trees to a fixed labeled tree. 

Let \(\mathcal{F} = \{T_i(l_i) \colon i \in I\}\) be a nonempty set of labeled trees \(T_i = T_i(l_{i})\) for which
\begin{equation}\label{e4.34}
V(T_{i_1}) \cap V(T_{i_2}) = \varnothing
\end{equation}
holds for all distinct \(i_1\), \(i_2 \in I\), and let \(T^{*} = T^{*}(l^{*})\) be a labeled tree such that \(V(T_i) \cap V(T^{*})\) is a single-point set \(\{v_i\}\),
\begin{equation}\label{e4.35}
V(T_i) \cap V(T^{*}) = \{v_i\}
\end{equation}
for every \(i \in I\). Let us suppose also
\begin{equation}\label{e4.36}
l^{*}(v_i) = l_{i}(v_i)
\end{equation}
for every \(i \in I\) if \(v_i\) satisfies \eqref{e4.35}. Then we define the gluing \(\mathcal{F}\) to \(T^{*}\) as a labeled graph \(T = T(l)\) with 
\begin{equation}\label{e4.37}
V(T) \stackrel{\mathrm{def}}{=} V(T^{*}) \cup \left(\bigcup_{i \in I} V(T_i)\right), \quad E(T) \stackrel{\mathrm{def}}{=} E(T^{*}) \cup \left(\bigcup_{i \in I} E(T_i)\right)
\end{equation}
and \(l \colon V(T) \to \RR\) such that
\begin{equation}\label{e4.38}
l(v) = \begin{cases}
l^{*}(v) & \text{if } v \in V(T^{*}),\\
l_{i}(v) & \text{if } v \in V(T_i) \text{ for some } i \in I.
\end{cases}
\end{equation}
Using equalities \eqref{e4.34}--\eqref{e4.38} one can simply show that \(T = T(l)\) is a well-defined labeled tree for which the labeling \(l\) is non-degenerate if and only if all labelings \(l_i\), \(i \in I\), and \(l^{*}\) are non-degenerate.

\begin{example}\label{ex11.10}
Let \(R^{*} = R^{*}(l^{*})\) be a labeled ray such that \(V(R^{*}) = \mathbb{N}\) and
\[
\bigl(\{m, n\} \in E(R^{*})\bigr) \Leftrightarrow \bigl(|m-n|=1\bigr)
\]
for all \(m\), \(n \in \mathbb{N}\) and, let the equality
\[
l^{*}(m) = \begin{cases}
\frac{1}{m} & \text{if \(m\) is odd},\\
0 & \text{if \(m\) is even}
\end{cases}
\]
hold for each \(m \in \NN\). Moreover, for every even \(m \in \mathbb{N}\) we define a labeled star \(S_m(l_m)\) with a center \(c_m = m\) and suppose that the following conditions hold:
\[
\quad V(S_m) \cap \NN = \{m\}, \quad l_m(c_m) = 0,
\]
and the restriction \(l_m|_{V(S_m) \setminus \{c_m\}}\) of \(l_m\) on the set \(V(S_m) \setminus \{c_m\}\) is a bijection to the set \(\{\frac{1}{mn} \colon n \in \mathbb{N}\}\); and
\[
V(S_{m_1}) \cap V(S_{m_2}) = \varnothing
\]
holds for all distinct even \(m_1\), \(m_1 \in \mathbb{N}\). Then we can consider the labeled tree \(T = T(l)\) obtained by gluing the set 
\[
\{S_m(l_m) \colon m \in \NN \text{ and } m \text{ is even}\}
\]
to the labeled ray \(R^{*}(l^{*})\) (see Figure~\ref{fig10}). The ultrametric space \((V(T), d_l)\) is totally bounded by Theorem~\ref{t11.8} but not complete by Theorem~\ref{t4.7}.
\end{example}

\begin{figure}[htb]
\begin{center}
\begin{tikzpicture}[
level 1/.style={level distance=\levdist,sibling distance=24mm},
level 2/.style={level distance=\levdist,sibling distance=12mm},
level 3/.style={level distance=\levdist,sibling distance=6mm},
solid node/.style={circle,draw,inner sep=1.5,fill=black},
micro node/.style={circle,draw,inner sep=0.5,fill=black}]

\def\dx{2cm}
\draw (0, 0) -- (6*\dx, 0);
\node [solid node] at (0, 0) [label= below:\(\frac{1}{1}\)] {};
\node [solid node] at (\dx, 0) [label= below:\(0\)] {};
\node at (0, \dx) [label= above:\(T(l)\)] {};
\draw (\dx, 0) -- +(150:0.9*\dx) node [solid node, label= above:\(\frac{1}{2}\)] {};
\draw (\dx, 0) -- +(110:0.8*\dx) node [solid node, label= above:\(\frac{1}{4}\)] {};
\draw (\dx, 0) -- +(80:0.7*\dx) node [solid node, label= above:\(\frac{1}{6}\)] {};
\draw (\dx, 0) +(70:0.67*\dx) node [micro node] {};
\draw (\dx, 0) +(60:0.65*\dx) node [micro node] {};
\draw (\dx, 0) +(50:0.63*\dx) node [micro node] {};
\draw (\dx, 0) +(40:0.6*\dx) node [solid node, label= above right:\(\frac{1}{2n}\)] {};
\draw (\dx, 0) +(30:0.5*\dx) node [micro node] {};
\draw (\dx, 0) +(20:0.45*\dx) node [micro node] {};
\draw (\dx, 0) +(10:0.4*\dx) node [micro node] {};
\node at (\dx, \dx) [label= above:\(S_1(l)\)] {};

\node [solid node] at (2*\dx, 0) [label= below:\(\frac{1}{3}\)] {};

\node [solid node] at (3*\dx, 0) [label= below:\(0\)] {};
\draw (3*\dx, 0) -- +(150:0.9*\dx) node [solid node, label= above:\(\frac{1}{4}\)] {};
\draw (3*\dx, 0) -- +(110:0.8*\dx) node [solid node, label= above:\(\frac{1}{8}\)] {};
\draw (3*\dx, 0) -- +(80:0.7*\dx) node [solid node, label= above:\(\frac{1}{12}\)] {};
\draw (3*\dx, 0) +(70:0.67*\dx) node [micro node] {};
\draw (3*\dx, 0) +(60:0.65*\dx) node [micro node] {};
\draw (3*\dx, 0) +(50:0.63*\dx) node [micro node] {};
\draw (3*\dx, 0) +(40:0.6*\dx) node [solid node, label= above right:\(\frac{1}{4n}\)] {};
\draw (3*\dx, 0) +(30:0.5*\dx) node [micro node] {};
\draw (3*\dx, 0) +(20:0.45*\dx) node [micro node] {};
\draw (3*\dx, 0) +(10:0.4*\dx) node [micro node] {};
\node at (3*\dx, \dx) [label= above:\(S_2(l)\)] {};

\node [micro node] at (3.3*\dx, 0) {};
\node [micro node] at (3.5*\dx, 0) {};
\node [micro node] at (3.7*\dx, 0) {};

\node [solid node] at (4*\dx, 0) [label= below:\(\frac{1}{2m-1}\)] {};

\node [solid node] at (5*\dx, 0) [label= below:\(0\)] {};
\draw (5*\dx, 0) -- +(150:0.9*\dx) node [solid node, label= above:\(\frac{1}{2m}\)] {};
\draw (5*\dx, 0) -- +(110:0.8*\dx) node [solid node, label= above:\(\frac{1}{4m}\)] {};
\draw (5*\dx, 0) -- +(80:0.7*\dx) node [solid node, label= above:\(\frac{1}{6m}\)] {};
\draw (5*\dx, 0) +(70:0.67*\dx) node [micro node] {};
\draw (5*\dx, 0) +(60:0.65*\dx) node [micro node] {};
\draw (5*\dx, 0) +(50:0.63*\dx) node [micro node] {};
\draw (5*\dx, 0) +(40:0.6*\dx) node [solid node, label= above right:\(\frac{1}{2mn}\)] {};
\draw (5*\dx, 0) +(30:0.5*\dx) node [micro node] {};
\draw (5*\dx, 0) +(20:0.45*\dx) node [micro node] {};
\draw (5*\dx, 0) +(10:0.4*\dx) node [micro node] {};
\node at (5*\dx, \dx) [label= above:\(S_m(l)\)] {};

\node [micro node] at (5.3*\dx, 0) {};
\node [micro node] at (5.5*\dx, 0) {};
\node [micro node] at (5.7*\dx, 0) {};
\end{tikzpicture}
\end{center}
\caption{The tree \(T\) has \(\aleph_0\) vertices with degree \(\aleph_0\) and the ultrametric space \((V(T), d_l)\) is totally bounded but not complete.} 
\label{fig10}
\end{figure}
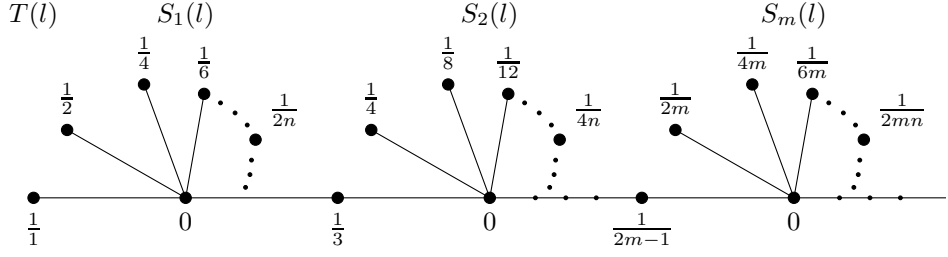

\begin{example}\label{ex4.21}
Let \(\mathbb{P}\) be the set of all integer prime numbers \(p \geqslant 2\) and let \(S^{*} = S^{*}(l^{*})\) be the labeled star with the vertex set \(V(S^{*}) = \{p \in \mathbb{P}\} \cup \{0\}\), and the center \(c^{*} = 0\), and the labeling \(l^{*} \colon V(S^{*}) \to \RR^{+}\) for which \(l^*(0) = 0\) and \(l^{*} (p) = 1/p\) hold for all \(p \in \mathbb{P}\).

For every \(p \in \mathbb{P}\) we define a labeled star \(S_p = S_p(l_p)\) with a center \(c_p\) such that:
\[
V(S_p) \setminus \{c_p\} = \{p^{n} \colon n \in \mathbb{N}\},
\]
and 
\[
c_p \notin \bigcup_{p \in \mathbb{P}} \bigl(V(S_p) \setminus \{c_p\}\bigr) \cup V(S^{*});
\]
\[
l_p(v) = \begin{cases}
0 & \text{if } v = c_p,\\
p^{-n} & \text{if } v = p^{n} \text{ for some } n \in \NN;
\end{cases}
\]
and \(c_{p_1} \neq c_{p_2}\) for all distinct \(p_1\), \(p_2 \in \mathbb{P}\). Then we obtain \(V(S^{*}) \cap V(S_p) = \{p\}\), and \(l^{*} (p) =  l_p(p) = 1/p\), and \(\delta_{S^{*}} (c^{*}) = \delta_{S_p} (c_p) = \aleph_0\) for every \(p \in \mathbb{P}\) .

Let us consider the labeled tree \(T = T(l)\) which is obtained by gluing the set \(\{S_p(l_p) \colon p \in \mathbb{P}\}\) to \(S^{*}(l^{*})\) (see Figure~\ref{fig1}), then the ultrametric space \((V(T), d_l)\) is compact by Theorem~\ref{t4.18}.
\end{example}

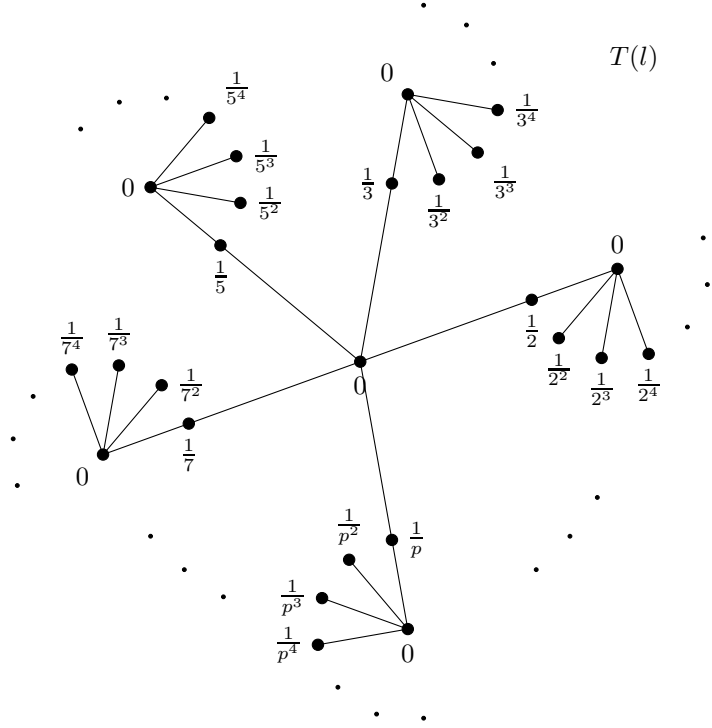
\begin{figure}[ht]
\begin{center}
\begin{tikzpicture}[
level 1/.style={level distance=\levdist,sibling distance=24mm},
level 2/.style={level distance=\levdist,sibling distance=12mm},
level 3/.style={level distance=\levdist,sibling distance=6mm},
solid node/.style={circle,draw,inner sep=1.5,fill=black},
micro node/.style={circle,draw,inner sep=0.5,fill=black}]

\def\dx{1.2cm}
\node at (3*\dx, 3*\dx) [label= above:\(T(l)\)] {};
\node [solid node] at (0, 0) [label= below:\(0\)] {};

\def\dphi{20}
\def\xx{2}
\draw (0, 0) -- (\dphi:3*\dx) node [solid node, label= above:\(0\)] {};
\node [solid node] at (\dphi:2*\dx) [label= below:\(\frac{1}{\xx}\)] {};
\draw (-180+\dphi+30:\dx)++(\dphi:3*\dx) node [solid node, label= below:\(\frac{1}{\xx^2}\)] {} -- (\dphi:3*\dx);
\draw (-180+\dphi+60:\dx)++(\dphi:3*\dx) node [solid node, label= below:\(\frac{1}{\xx^3}\)] {} -- (\dphi:3*\dx);
\draw (-180+\dphi+90:\dx)++(\dphi:3*\dx) node [solid node, label= below:\(\frac{1}{\xx^4}\)] {} -- (\dphi:3*\dx);
\node [micro node] at ($(-180+\dphi+120:\dx) + (\dphi:3*\dx)$) {};
\node [micro node] at ($(-180+\dphi+150:\dx) + (\dphi:3*\dx)$) {};
\node [micro node] at ($(-180+\dphi+180:\dx) + (\dphi:3*\dx)$) {};

\def\dphi{80}
\def\xx{3}
\draw (0, 0) -- (\dphi:3*\dx) node [solid node, label= above left:\(0\)] {};
\node [solid node] at (\dphi:2*\dx) [label= left:\(\frac{1}{\xx}\)] {};
\draw (-180+\dphi+30:\dx)++(\dphi:3*\dx) node [solid node, label= below:\(\frac{1}{\xx^2}\)] {} -- (\dphi:3*\dx);
\draw (-180+\dphi+60:\dx)++(\dphi:3*\dx) node [solid node, label= below right:\(\frac{1}{\xx^3}\)] {} -- (\dphi:3*\dx);
\draw (-180+\dphi+90:\dx)++(\dphi:3*\dx) node [solid node, label= right:\(\frac{1}{\xx^4}\)] {} -- (\dphi:3*\dx);
\node [micro node] at ($(-180+\dphi+120:\dx) + (\dphi:3*\dx)$) {};
\node [micro node] at ($(-180+\dphi+150:\dx) + (\dphi:3*\dx)$) {};
\node [micro node] at ($(-180+\dphi+180:\dx) + (\dphi:3*\dx)$) {};

\def\dphi{140}
\def\xx{5}
\draw (0, 0) -- (\dphi:3*\dx) node [solid node, label= left:\(0\)] {};
\node [solid node] at (\dphi:2*\dx) [label= below:\(\frac{1}{\xx}\)] {};
\draw (-180+\dphi+30:\dx)++(\dphi:3*\dx) node [solid node, label= right:\(\frac{1}{\xx^2}\)] {} -- (\dphi:3*\dx);
\draw (-180+\dphi+60:\dx)++(\dphi:3*\dx) node [solid node, label= right:\(\frac{1}{\xx^3}\)] {} -- (\dphi:3*\dx);
\draw (-180+\dphi+90:\dx)++(\dphi:3*\dx) node [solid node, label= above right:\(\frac{1}{\xx^4}\)] {} -- (\dphi:3*\dx);
\node [micro node] at ($(-180+\dphi+120:\dx) + (\dphi:3*\dx)$) {};
\node [micro node] at ($(-180+\dphi+150:\dx) + (\dphi:3*\dx)$) {};
\node [micro node] at ($(-180+\dphi+180:\dx) + (\dphi:3*\dx)$) {};

\def\dphi{200}
\def\xx{7}
\draw (0, 0) -- (\dphi:3*\dx) node [solid node, label= below left:\(0\)] {};
\node [solid node] at (\dphi:2*\dx) [label= below:\(\frac{1}{\xx}\)] {};
\draw (-180+\dphi+30:\dx)++(\dphi:3*\dx) node [solid node, label= right:\(\frac{1}{\xx^2}\)] {} -- (\dphi:3*\dx);
\draw (-180+\dphi+60:\dx)++(\dphi:3*\dx) node [solid node, label= above:\(\frac{1}{\xx^3}\)] {} -- (\dphi:3*\dx);
\draw (-180+\dphi+90:\dx)++(\dphi:3*\dx) node [solid node, label= above:\(\frac{1}{\xx^4}\)] {} -- (\dphi:3*\dx);
\node [micro node] at ($(-180+\dphi+120:\dx) + (\dphi:3*\dx)$) {};
\node [micro node] at ($(-180+\dphi+150:\dx) + (\dphi:3*\dx)$) {};
\node [micro node] at ($(-180+\dphi+180:\dx) + (\dphi:3*\dx)$) {};

\node [micro node] at (220:3*\dx) {};
\node [micro node] at (230:3*\dx) {};
\node [micro node] at (240:3*\dx) {};

\def\dphi{280}
\def\xx{p}
\draw (0, 0) -- (\dphi:3*\dx) node [solid node, label= below:\(0\)] {};
\node [solid node] at (\dphi:2*\dx) [label= right:\(\frac{1}{\xx}\)] {};
\draw (-180+\dphi+30:\dx)++(\dphi:3*\dx) node [solid node, label= above:\(\frac{1}{\xx^2}\)] {} -- (\dphi:3*\dx);
\draw (-180+\dphi+60:\dx)++(\dphi:3*\dx) node [solid node, label= left:\(\frac{1}{\xx^3}\)] {} -- (\dphi:3*\dx);
\draw (-180+\dphi+90:\dx)++(\dphi:3*\dx) node [solid node, label= left:\(\frac{1}{\xx^4}\)] {} -- (\dphi:3*\dx);
\node [micro node] at ($(-180+\dphi+120:\dx) + (\dphi:3*\dx)$) {};
\node [micro node] at ($(-180+\dphi+150:\dx) + (\dphi:3*\dx)$) {};
\node [micro node] at ($(-180+\dphi+180:\dx) + (\dphi:3*\dx)$) {};

\node [micro node] at (310:3*\dx) {};
\node [micro node] at (320:3*\dx) {};
\node [micro node] at (330:3*\dx) {};
\end{tikzpicture}
\end{center}
\caption{The tree \(T\) has \(\aleph_0\) vertices with degree \(\aleph_0\) and the ultrametric space \((V(T), d_l)\) is compact.} 
\label{fig1}
\end{figure}

The following simple example shows that the class of finite ultrametric spaces, which are representable in the form \((V(T), d_l)\), is a proper subclass of all finite ultrametric spaces.

\begin{example}\label{ex11.11}
Let \(V = \{v_1, v_2, v_3, v_4\}\) be a four-point set and let an ultrametric \(d \colon V \times V \to \RR^{+}\) satisfy the equalities
\begin{gather}\label{ex11.11:e1}
d(v_1, v_2) = d(v_3, v_4) = 1\\
\intertext{and}
\label{ex11.11:e2}
d(v_1, v_3) = d(v_1, v_4) = d(v_2, v_3) = d(v_2, v_4) = 2.
\end{gather}
Let \(T = T(l)\) be a labeled tree such that \(V(T) = V\). We claim that the ultrametric spaces \((V, d)\) and \((V(T), d_l)\) are not isometric for any non-degenerate labeling \(l\). Indeed, if there is a non-degenerate \(l \colon V(T) \to \RR^{+}\) for which \((V, d)\) and \((V(T), d_l)\) are isomorphic, then from~\eqref{e11.3} and \eqref{ex11.11:e1} it follows that
\[
\max_{1 \leqslant i \leqslant 4} l(v_i) \leqslant 1.
\]
The last inequality and \eqref{e11.3} imply that \(d_l(v, u) \leqslant 1\) holds for all \(u\), \(v \in V\), contrary to \eqref{ex11.11:e2}.
\end{example}

It seems to be interesting to get a purely metric characterization of ultrametric spaces generated by labeled trees.

\begin{conjecture}\label{con11.1}
Let \((X, d)\) be a discrete nonempty totally bounded ultrametric space. Then the following statements are equivalent:
\begin{enumerate}
\item\label{con11.1:s1} There is a labeled tree \(T = T(l)\) such that \((V(T), d_l)\) and \((X, d)\) are isometric.
\item\label{con11.1:s2} For every \(B \in \BB_{X}\), there are \(c \in X\) and \(r > 0\) such that
\[
B = \{x \in X \colon d(x, c) = r\} \cup \{c\}
\]
i.e., the ball \(B\) is the sphere \(S(c, r) = \{x \in X \colon d(x, c) = r\}\) with the added center \(c\).
\end{enumerate}
\end{conjecture}

In conclusion, we formulate a simple conjecture linking the properties of Cauchy sequences in \((V(T), d_l)\) with the structure of the hull of the range sets of these sequences (cf. Lemma~\ref{l4.8}).

\begin{conjecture}\label{con5.2}
Let \(T\) be a tree and let \((v_n)_{n \in \NN}\) be a sequence of distinct vertices of \(T\). Then the following conditions are equivalent:
\begin{enumerate}
\item The hull of the range set of \((v_n)_{n \in \NN}\) is a union of a ray with some finite tree.
\item For every non-degenerate \(l \colon V(T) \to \RR\) the existence of Cauchy subsequence in \((v_n)_{n \in \NN}\) implies that \((v_n)_{n \in \NN}\) is also a Cauchy sequence.
\end{enumerate}
\end{conjecture}


\providecommand{\bysame}{\leavevmode\hbox to3em{\hrulefill}\thinspace}
\providecommand{\MR}{\relax\ifhmode\unskip\space\fi MR }
\providecommand{\MRhref}[2]{%
  \href{http://www.ams.org/mathscinet-getitem?mr=#1}{#2}
}
\providecommand{\href}[2]{#2}

\end{document}